\documentclass{amsart}

\usepackage[margin=3cm]{geometry}
\usepackage{graphics, amsmath,enumitem, comment, mathtools}
\usepackage{graphicx}
\usepackage{epsfig}

\newif\ifcolorcomments
\newcommand{\allowcomments}[4]{
\newcommand{#1}[1]{\ifdraft{\ifcolorcomments{\textcolor{#4}{##1 --#3}}\else{\textsl{ ##1 \ --#3}}\fi}\else{}\fi}
}

\allowcomments{\commumtaz}{MH}{Mumtaz}{green}
\allowcomments{\comwang}{BW}{BWWang}{blue}
\allowcomments{\comkle}{DK}{DK}{magenta}
\allowcomments{\comnick}{NW}{Nick}{red}

\newcommand {\ignore}[1] {}

\colorcommentstrue
\usepackage{amssymb}
\usepackage{amsmath,amsthm}
\usepackage{mathrsfs}
\usepackage{bbm}
\def\bc{\begin{center}}
\def\ec{\end{center}}
\def\be{\begin{equation}}
\def\ee{\end{equation}}

\def\N{\mathbb N}
\def\Z{\mathbb Z}
\def\Q{\mathbb Q}

\def\R{\mathbb R}
\def\A{\mathbb A}
\def\D{\mathcal E}
\def\H{\mathcal H}

\newcommand\hdim{\dim_{\mathrm H}}

\def\Z{\mathbb Z}
\def\U{\mathcal U}

\newtheorem{lem}{Lemma}[section]

\newtheorem{pro}[lem]{Proposition}
\newtheorem{thm}[lem]{Theorem}

\newtheorem{cor}[lem]{Corollary}
\newtheorem{rem}[lem]{Remark}
\numberwithin{equation}{section}
\usepackage{colortbl}

\newif\ifdraft\drafttrue

\newcommand*{\myDots}{\ifmmode\mathellipsis\else.\kern-0.07em.\kern-0.07em.\fi}
\DeclarePairedDelimiter\floor{\lfloor}{\rfloor}

\newcommand{\subscript}[2]{$#1 _ #2$}

\begin{document}


\title[weighted products of multiple  partial quotients]{Metrical properties for the weighted products of \\ multiple  partial quotients in continued fractions}

\author[A. Bakhtawar]{Ayreena Bakhtawar}
\address{Ayreena Bakhtawar, School of Mathematics and Statistics, University of New South Wales, Sydney, NSW 2052, Australia}
\email{a.bakhtawar@unsw.edu.au}

\author[M. Hussain]{Mumtaz ~Hussain}
\address{Mumtaz Hussain,  Department of Mathematical and Physical Sciences,  La Trobe University, Bendigo 3552, Australia. }\email{m.hussain@latrobe.edu.au}
\author[D. Kleinbock]{Dmitry ~Kleinbock}
\address {Dmitry Kleinbock, Brandeis University, Waltham MA 02454-9110.} \email{kleinboc@brandeis.edu}
\author[B-W. Wang]{Bao-Wei Wang}
\address{Bao-wei Wang, School  of  Mathematics  and  Statistics,  Huazhong  University  of Science  and  Technology, 430074 Wuhan, China}
 \email{bwei\_wang@hust.edu.cn}

\date{February 2022}

\maketitle

\begin{abstract}
The classical Khintchine and Jarn\'ik theorems, generalizations of a consequence of Dirichlet's theorem,
are fundamental results in the theory of Diophantine approximation. These theorems are concerned
with the size of the set of real numbers for which the partial quotients in their continued fraction expansions grows
with a certain rate. Recently it was observed that
the growth of product of pairs of consecutive partial quotients in the continued fraction expansion of a real number is associated with 
improvements to Dirichlet's theorem.  In this paper we consider 
the products of several
consecutive partial quotients raised to different powers.
 Namely, we find the Lebesgue measure and the Hausdorff dimension of the following set:
$$
{\D_{\mathbf t}}(\psi):=\left\{x\in[0, 1):  
\prod\limits_{i=0}^{m-1}{a^{t_i}_{n+i}(x)}
\ge \Psi(n)\ {\text{for infinitely many}} \ n\in \N        
\right\},
$$
where $t_i\in\mathbb R_+$ for all ${0\leq i\leq m-1}$, and $\Psi:\N\to\R_{\ge 1}$ is a positive function.

\end{abstract}


\section{Statement of Results}

The fundamental objective in the theory of Diophantine approximation is to seek answers to the question \emph {how well an irrational number can be approximated by a rational number}? This question in the one dimensional settings has been well understood  as the theory of continued fractions provides quick and efficient way for finding good rational approximations to irrational numbers.  The continued fraction can be computed by the Guass transformation $T:[0, 1)\to [0, 1)$ defined as
\begin{equation*}
T(0) =0\quad \quad {\rm and} \quad T(x) =\frac{1}{x}({\rm mod} \ 1) \quad {\rm if} \quad x\in(0,1).
\end{equation*}
Then every $x\in[0,1)$ admits a unique continued fraction expansion
\begin{equation*}\label{cf1}
x=\frac{1}{a_{1}(x)+\displaystyle{\frac{1}{a_{2}(x)+\displaystyle{\frac{1}{
a_{3}(x)+_{\ddots }}}}}}\end{equation*}
where
$a_{n}(x)$ are called the partial quotients of $x$ with $$a_{1}(x)=\floor*{\frac{1}{x}}  \quad \text{and}\quad  a_{n}(x)=\floor*{ \frac{1} {T^{n}(x)}}=a_1\left(T^{n-1}(x)\right)\in\N$$ for each $n\geq 1$ (where $\lfloor \cdot\rfloor$ stands for the integral part). Equation \eqref{cf1} can also be represented as $$x=[a_{1}(x),a_{2}(x),a_{3}(x),\ldots,a_{n}(x)+T^{n}x ]= [a_{1}(x),a_{2}(x),a_{3}(x),\ldots ].$$

Studying the properties of growth of partial quotients valid for almost all (or almost none) $x\in[0, 1)$  is a major area of investigation within the theory of continued fractions and is referred to as the metrical theory of continued fractions. Since the partial quotients can be obtained through Gauss map, the theory has close connections with dynamical systems, ergodic theory and Diophantine approximation. Historically, the focus has been on the metrical theory of the sets
\begin{equation*}\label{BB}
\mathcal E(\Psi):=\left\{x\in[0, 1): a_{n}(x)\geq \Psi(n) \ \text{for infinitely many} \ n\in \N \right\} \end{equation*}
for a given function $\Psi:\N\to\R_{\ge 1}$.
Borel-Bernstein's theorem \cite{Be_12, Bo_12} is a fundamental result that describes the size of the set $\mathcal E(\Psi)$ in terms of  Lebesgue measure.
\begin{thm}[Borel-Bernstein, 1911-1912]\label{Bor}  Let $\Psi:\N\to\R_{\ge 1}$. Then
\begin{equation*}
\mathcal L\big(\mathcal E(\Psi)\big)=
\begin{cases}
0\  & \mathrm{if}\quad \sum_{n=1}^{\infty }\frac{1}{\Psi (n)} \,<\,\infty, \\[2ex]
1 \  & \mathrm{if}\quad \sum_{n=1}^{\infty}\frac{1}{\Psi (n)} \, =\,\infty .
\end{cases}
\end{equation*}
\end{thm}
Good \cite{Good} and {\L}uczak \cite{Luczak} were the main contributors to studying the Hausdorff dimension of this set for $\Psi(n)$ tending to
infinity at a polynomial $n^{a}$ and super-exponential speeds $a^{b^n}$ respectively,  see also \cite{MR1464376, Hir, Moor} and references therein.
Then the dimension of $\mathcal{E}(\Psi)$ was computed by Wang-Wu \cite{WaWu08} for arbitrary $\Psi$.
In what follows, $P(T, \phi)$ will stand for the pressure function for the dynamics of the Gauss map $T$ with potential $\phi$; see \S\ref{Pressure Functions} for a precise definition.
\begin{thm}[Wang-Wu, 2008]
\label{WaWu}
Let $\Psi:\N\to\R_{\ge 1}$. Denote \begin{equation}\label{Bb}\log B:=\liminf\limits_{n\rightarrow \infty }\frac{\log
\Psi (n)}{n} \ { and}\ \log b:=\liminf\limits_{n\rightarrow \infty }\frac{\log
\log \Psi (n)}{n}.\end{equation}
Then
\begin{equation*}
\hdim \mathcal E(\Psi)=\left\{
\begin{array}{ll}1,&
\mathrm{if}\ \  B=1, \\ [2ex]
\inf \left\{s\geq 0: P\big(T, -s(\log B+\log |T^{\prime
}|)\big)\le 0\right\}&
\mathrm{if}\ \ 1<B< \infty, \\ [2ex] \frac{1}{1+b} &
\mathrm{if}\ \ B=\infty.
\end{array}
\right.
\end{equation*} {In particular, $\hdim \mathcal E(\Psi) > 1/2$ if $B<\infty$.}
\end{thm}


In this paper, we study  a generalized form of the set $\mathcal E(\Psi)$ which has close connections with the improvements to Dirichlet's theorem (1842).
Namely, in \cite{KleinbockWadleigh}  Kleinbock-Wadleigh considered 
the set
\begin{equation}\label{E2}
\mathcal E_2(\Psi):=\left\{x\in[0, 1): {a_{n}(x)a_{n+1}(x)}\geq \Psi(n) \ \text{for infinitely many} \ n\in \N \right\},
\end{equation}
and found a zero-one law for $\mathcal L\big(\mathcal E_2(\Psi)\big)$, see \cite[Theorem 3.6]{KleinbockWadleigh}. This result was used to establish a zero-one law for the sets of $\psi$-Dirichlet improvable real numbers \cite[Theorem 1.8]{KleinbockWadleigh}, where $\psi$ is a positive non-increasing function. See  \cite[\S2]{KleinbockWadleigh} for a connection between \eqref{E2} and the improvements to Dirichlet's theorem, and {\cite{HKWW, BHS, Ba_20}} for further results in that direction.

The work of Kleinbock-Wadleigh was followed by Huang-Wu-Xu \cite{HuWuXu} with both Lebesgue measure and Hausdorff dimension results for a natural generalization of the set \eqref{E2}. Namely, for $m\in\N$ they considered
\begin{equation}\label{em}
\mathcal E_m(\Psi):=\left\{x\in[0, 1): {a_{n}(x)\cdots a_{n+m-1}(x)}\geq \Psi(n) \ \text{for infinitely many} \ n\in \N \right\},
\end{equation}
and proved the following
\pagebreak
\begin{thm}[Huang-Wu-Xu, 2019]\label{HWXLeb} Given  {$\Psi:\N\to\R_{\ge 1}$},
\begin{itemize}
\item[\rm (a)] \cite[Theorem 1.5]{HuWuXu}
\begin{equation*}
\mathcal L\big(\mathcal E_m(\Psi)\big)=
\begin{cases}
0\  & \mathrm{if}\quad \sum\limits_{n}^\infty\frac {\log^{m-1}\Psi(n)}{\Psi(n)} \,<\,\infty ,\\[2ex]
1 \  & \mathrm{if}\quad \sum\limits_{n}^\infty\frac {\log^{m-1}\Psi(n)}{\Psi(n)}\,=\,\infty ;
\end{cases}
\end{equation*}
\item[\rm (b)] \cite[Theorem 1.7]{HuWuXu}
\begin{equation*}
\dim _{\mathrm{H}}\mathcal{E} _m(\Psi )=\left\{
\begin{array}{ll}1,&
\mathrm{if}\ \  B=1, \\ [3ex]
\inf \big\{s\geq 0: {P}\left( T,-f_m(s)\log B-s\log|T^{\prime
}|\right) \leq 0\big\}&
\mathrm{if}\ \ 1<B< \infty; \\ [3ex] \frac{1}{1+b} &
\mathrm{if}\ \ B=\infty,
\end{array}
\right.
\end{equation*}
where   $B,b$ are as in \eqref{Bb},
and $f_m$ is given by the following iterative formula:
\begin{equation}\label{fi}{f_1(s)=s, \quad f_{k+1}(s)=\frac{sf_k(s)}{1-s+f_k(s)}, \ k\geq 1.}\end{equation}
\end{itemize}
\end{thm}
In this paper we consider a weighted generalization of \eqref{em}:  take ${\mathbf t} = {(t_0,\dots,t_{m-1})}\in\mathbb R_+^m$ and  {$\Psi:\N\to\R_{\ge 1}$}, and define
\begin{equation*}\label{mainset}
{\D_{\mathbf t}}(\Psi):=\left\{x\in[0, 1):  {\prod_{i=0}^{m-1}}a^{t_i}_{n+i}(x)\ge \Psi(n) \ {\text{for infinitely many}} \ n\in \N \right\}.
\end{equation*}
Clearly {$\D_m (\Psi)=\D_{{\mathbf 1}_m}(\Psi)$, where ${\mathbf 1}_m = (\underbrace{1,\dots,1}_{m})$}.
Generalizing Theorem \ref{HWXLeb}(a), we prove the following dichotomy statement for the Lebesgue measure of ${\D_{\mathbf t}}(\Psi)$:

\begin{thm}\label{mainLeb} Let  {$\Psi:\N\to\R_{\ge 1}$}. Then
\begin{equation*}
\mathcal L\big({\D_{\mathbf t}}(\Psi )\big)=\left\{
\begin{array}{ll}0,&
\mathrm{if}\ \  \sum\limits_{n=1}^{\infty}\frac{(\log \Psi(n))^{\ell-1}}{\Psi(n)^{1/t_{\max}}} <\infty, \\ [2ex]
1 &
\mathrm{if}\ \  \sum\limits_{n=1}^{\infty}\frac{(\log \Psi(n))^{\ell-1}}{\Psi(n)^{1/t_{\max}}} =\infty,\end{array}
\right.
\end{equation*}where \begin{equation}\label{tmax}
t_{\max}=\max\{t_i: {0\le i\le m-1}\}, \ \ell=\# \{i: t_i=t_{\max}\}.
\end{equation}
\end{thm}

\ignore{\begin{cor}\label{KKW}
Let $\Psi:\mathbb N\to [1, \infty)$ be an increasing function. Define
 $$
F_m^{\mathbf t} (\Psi):=\left\{x\in[0, 1):  \prod_{i=1}^ma_{n+i}^{t_i}(x) \ge \Psi(q_n(x)) \ {\text{for i.m.}} \ n\in \N\right\}.
$$
Then
\begin{equation*}
\mathcal L\left(F^\mathbf t_m(\Psi)\right)=
\begin{cases}
0\  & \mathrm{if}\quad \sum\limits_{n=1}^{\infty}\frac{(\log \Psi(n))^{\ell-1}}{n\Psi(n)^{1/t_{\max}}}<\infty ; \\[2ex]
1 \  & \mathrm{if}\quad \sum\limits_{n=1}^{\infty}\frac{(\log \Psi(n))^{\ell-1}}{n\Psi(n)^{1/t_{\max}}}=\infty .
\end{cases}
\end{equation*}
\end{cor}

Corollary \ref{KKW}  recovers not only Khintchine's classical theorem (1924) by setting $m=1$, and $t_1=t_2=\ldots=1$ but also  Kleinbock-Wadleigh theorem \cite{KleinbockWadleigh}  by setting $m=2$, and $t_1=t_2=\ldots=1$. See the next section for further details.}

A weighted generalization of Theorem \ref{HWXLeb}(b)  
{is straightforward in} the
{case when $B$ is either infinite or equal to $1$:
\begin{thm}\label{BHWthmeasy}  Let $\Psi:\N\to\R_{\ge 1}$, and let $B,b$ be as in \eqref{Bb}.
Then
\begin{equation*}
\hdim {\D_{\mathbf t}}(\Psi )=\left\{
\begin{array}{ll} \ \, 1&
\mathrm{if}\ \  B=1, \\ [3ex]
\frac{1}{1+b} &
\mathrm{if}\ \ B=\infty.  \\ [3ex]
\end{array}
\right.
\end{equation*}
\end{thm}}

As for the remaining intermediate case $1<B< \infty$, we are only able to treat the $m=2$ case, characterizing the Hausdorff dimension of sets ${\D_{\mathbf t}}(\Psi )$ for ${\mathbf t} = ({t_0,t_1})\in\mathbb R_+^2$.
\begin{thm}\label{BHWthm}  Let $\Psi:\N\to\R_{\ge 1}$ be such that $1<B< \infty$, and let ${\mathbf t} = ({t_0,t_1})\in\mathbb R_+^2$.
Then
\begin{equation*}
\hdim {\D_{\mathbf t}}(\Psi )=
 \inf \big\{ s\geq 0: P(T, -s\log |T'|-f_{{t_0,t_1}}(s)\log B)\le 0\big\}
,
\end{equation*}
where 
\begin{equation}\label{ft1t2} f_{{t_0,t_1}}(s):=\frac{s^2}{{t_0t_1}\cdot \max\left\{{\frac{s}{t_1}+\frac{1-s}{t_0}, \frac{s}{t_0}}\right\}}.\end{equation}
\end{thm}

Note that $f_{1, 1}(s) = s^2$ {for all $0 \le s \le 1$}, which agrees with the $k=2$ case of \eqref{fi}. See \S\ref{final} for an explanation of why the case $m > 2$ is much more involved.

\begin{rem} It is worth highlighting an interesting phenomena here. The Lebesgue measure of the set ${\mathcal E_{{\mathbf t}}(\Psi )}$ is independent of the ordering of the exponents, whereas the Hausdorff dimension depends on it. For instance
\begin{equation*}\label{f12} f_{2, 1}(s)=\frac {s^2}{1+s},\ {\text{and}}\ \ f_{1, 2}(s)=\begin{cases}\frac{s^2}{2-s}&\text{\rm if }s \le \frac23;\\ \frac{s}{2}&\text{\rm if }s > \frac23.\end{cases}\end{equation*}
It is easy to see that $f_{2,1}(s)<f_{1,2}(s)$ for any $1/2<s<1$. Since  $\hdim \D_{\mathbf (1,2)}(\Psi ) \ge \hdim \D(\Psi ) > 1/2$
whenever $B<\infty$ (see Theorem \ref{WaWu}), 
it follows that in  Theorem \ref{BHWthm} 
one always {has}
$${\hdim \D_{\mathbf (2,1)}(\Psi )>\hdim \D_{\mathbf (1,2)}(\Psi )}.$$
\end{rem}

\noindent{\bf Acknowledgements.} {The research of A.\ Bakhtawar is supported by the ARC grant DP180100201, of M.\ Hussain 
by the ARC grant DP200100994, of D.\ Kleinbock 
by the NSF grant DMS-1900560, and of B.\ Wang by NSFC (11831007).}  Part of this work was carried out during the workshop ``Ergodic Theory, Diophantine approximation and related topics'' sponsored by the MATRIX Research Institute.

\section{Preliminaries and auxiliary results}\label{S2}

For completeness we give a brief introduction to Hausdorff measures and dimension.  For further details we refer to the beautiful texts \cite{BernikDodson, Falconer_book}.

\subsection{Hausdorff measure and dimension}\label{HM}\

Let $0<s\in\R^n$ let
$E\subset \R^n$.
 Then, for any $\rho>0$ a countable collection $\{B_i\}$ of balls in
$\R^n$ with diameters $\mathrm{diam} (B_i)\le \rho$ such that
$E\subset \bigcup_i B_i$ is called a $\rho$-cover of $E$.
Let
\[
\H_\rho^s(E)=\inf \sum_i \mathrm{diam}(B_i)^s,
\]
where the infimum is taken over all possible $\rho$-covers $\{B_i\}$ of $E$. It is easy to see that $\H_\rho^s(E)$ increases as $\rho$ decreases and so approaches a limit as $\rho \rightarrow 0$. This limit could be zero or infinity, or take a finite positive value. Accordingly, the \textit{$s$-Hausdorff measure $\H^s$} of $E$ is defined to be
\[
\H^s(E)=\lim_{\rho\to 0}\H_\rho^s(E).
\]
It is easily verified that Hausdorff measure is monotonic and countably sub-additive, and that $\H^s(\varnothing)=0$. Thus it is an outer measure on $\R^n$.
When $s=n$,  $\H^n$ coincides with standard Lebesgue measure on $\R^n$.

For any subset $E$ one can verify that there exists a unique critical value of $s$ at which $\H^s(E)$ `jumps' from infinity to zero. The value taken by $s$ at this discontinuity is referred to as the \textit{Hausdorff dimension of  $E$} and  is denoted by $\hdim E $; i.e.,
\[
\hdim E :=\inf\left\{s\in \R_+\;:\; \H^s(E)=0\right\}.
\] 

Computing Hausdorff dimension of a set is typically accomplished in two steps: obtaining the upper and lower bounds separately.
Upper bounds often can be handled by finding appropriate coverings. When dealing with a limsup set, one 
 {usually applies} the Hausdorff measure version of the famous Borel-Cantelli lemma (see Lemma 3.10 of \cite{BernikDodson}):

\begin{pro}\label{bclem}
    Let $\{B_i\}_{i\ge 1}$ be a sequence of measurable  sets in $\R$ and suppose that,  $$\sum_i \mathrm{diam}(B_i)^s \, < \, \infty.$$ Then  $$\H^s({\limsup_{i\to\infty}B_i})=0.$$
\end{pro}

\subsection{Continued fractions and Diophantine approximation}
 \

Suppose that $x \in [0,1)\smallsetminus \Q$ has continued fraction expansion $x= [a_1, a_2,\dots]$,
where $a_n(x)=\lfloor 1/T^{n-1}(x)\rfloor$
for each $n\ge 1$.
Recall the sequences $p_n= p_n(x)$, $q_n= q_n(x)$ 
defined by the recursive relation $(p_{-1},q_{-1}) = (0,1)$, $(p_{0},q_{0}) = (1,1)$, and
\begin {equation}\label{recu}
p_{n+1}=a_{n+1}(x)p_n+p_{n-1}, \ \
q_{n+1}=a_{n+1}(x)q_n+q_{n-1},\ \  n\geq 0.
\end {equation}
Thus $p_n=p_n(x), q_n=q_n(x)$ are determined by the partial quotients $a_1,\dots,a_n$, so we may write  \linebreak $p_n=p_n(a_1,\dots, a_n)$, $q_n=q_n(a_1,\dots,a_n)$. When it is clear which partial quotients are involved, we denote them by $p_n, q_n$ for simplicity.

For any integer vector $(a_1,\dots,a_n)\in \N^n$ with $n\geq 1$, write
\begin{equation*}\label{cyl}
I_n(a_1,\dots,a_n):=\left\{x\in [0, 1): a_1(x)=a_1, \dots, a_n(x)=a_n\right\}
\end{equation*}
for the corresponding `cylinder of order $n$', i.e.\  the set of all real numbers in $[0,1)$ whose continued fraction expansions begin with $(a_1, \dots, a_n).$
We will use $I_n(x)$ to denote the $n$th order cylinder containing $x$.

We will frequently use the following well known properties of continued fraction expansions.  They are explained in the standard texts \cite{IosKra_book, Khi_63}.

\begin{pro}\label{pp3} For any {positive} integers $a_1,\dots,a_n$, let $p_n=p_n(a_1,\dots,a_n)$ and $q_n=q_n(a_1,\dots,a_n)$ be defined recursively by \eqref{recu}. {Then:}
\begin{enumerate}[label={\rm (\subscript{\rm P}{\arabic*})}]
\item
\begin{eqnarray*}
I_n(a_1,a_2,\dots,a_n)= \left\{
\begin{array}{ll}
         \left[\frac{p_n}{q_n}, \frac{p_n+p_{n-1}}{q_n+q_{n-1}}\right)     & {\rm if }\ \
         n\ {\rm{is\ even}};\\
         \left(\frac{p_n+p_{n-1}}{q_n+q_{n-1}}, \frac{p_n}{q_n}\right]     & {\rm if }\ \
         n\ {\rm{is\ odd}}.
\end{array}
        \right.
\end{eqnarray*}
{\rm Thus, its length is given by}
\begin{equation*}\label{lencyl}
\frac{1}{2q_n^2}\leq |I_n(a_1,\ldots,a_n)|=\frac{1}{q_n(q_n+q_{n-1})}\leq \frac1{q_n^2},
\end{equation*}
{\rm since} $$
 p_{n-1}q_n-p_nq_{n-1}=(-1)^n, \ {\rm for \ all }\ n\ge 1.
 $$

\item For any $n\geq 1$, $q_n\geq 2^{(n-1)/2}$ and
$$
1\le \frac{q_{n+m}(a_1,\dots, a_n, b_1,\dots, b_m)}{q_n(a_1,\dots, a_n)\cdot q_m(b_1,\dots,b_m)}\le 2.
$$
\item $$\prod_{i=1}^na_i\leq q_n\leq \prod_{i=1}^n(a_i+1)\leq 2^n\prod_{i=1}^na_i.$$

\item there exists a constant $K>1$ such that for almost all $x\in [0,1)$, $$
q_n(x)\le K^n, \ {\text{for all $n$ sufficiently large}}.
$$
\end{enumerate}
\end{pro}

Let $\mu_G$ be the Gauss measure given by $$
d\mu_G=\frac{1}{\log 2}\cdot \frac{1}{(1+x)}\,dx.
$$ It is known that $\mu_G$ is $T$-invariant; clearly it is equivalent to Lebesgue measure $\mathcal{L}$.

The next proposition concerns the position of a cylinder in $[0,1)$.
\begin{pro}[Khintchine, 1963]\label{pp2} Let $I_n=I_n(a_1,\dots, a_n)$ be a cylinder of order $n$, which is partitioned into sub-cylinders $\{I_{n+1}(a_1,\dots,a_n, a_{n+1}): a_{n+1}\in \N\}$. When $n$ is odd, these sub-cylinders are positioned from left to right, as $a_{n+1}$ increases from 1 to $\infty$; when $n$ is even, they are positioned from right to left.
\end{pro}

The following result is due to {\L}uczak \cite{Luczak}.
\begin{lem}[{\L}uczak, 1997]\label{lemb}For any $b, c>1$, the sets
\begin{align*}
&\left\{x\in[0, 1):  a_{n}(x)\ge c^{b^n}\  {\text{for infinitely many}} \ n\in \N       \right\},\\
&\left\{x\in[0, 1):  a_{n}(x)\ge c^{b^n}\  {\text{for all }} \ n\geq 1 \right\},
\end{align*}
have the same Hausdorff dimension $\frac1{b+1}$.
\end{lem}

\subsection{Pressure function and Hausdorff dimension} \label{Pressure Functions}\

In this section, we recall a fact that the pressure function with a continuous potential can be approximated by the pressure functions restricted to
the sub-systems in continued fractions. For more thorough results on pressure function in infinite conformal iterated function systems, see Hanus-Mauldin-Urba\'{n}ski \cite{H02}, Mauldin-Urba\'{n}ski \cite{M96,MU99}, or their monograph \cite{MU03}.

Let $\A$ be a finite or infinite subset of $\N$. Define
$$X_{\A}=\{x\in [0,1):\ a_n(x)\in \A,\ \text{for all} \ n\geq 1\}.$$
Then $(X_{\A},T)$ is a sub-system of $([0,1),T).$ Let $\phi: [0,1)\rightarrow \R$ be a real function. The pressure function restricted to the system $(X_{\A},T)$ with the potential $\phi$ is defined by
\begin{equation}\label{pre}P_{\A}(T, \phi)=\lim_{n\to \infty}\frac{1}{n}\log \sum_{(a_1,\ldots,a_n)\in \A^n}\sup_{x\in X_{\A}}e^{S_n\phi([a_1,\ldots,a_n+x])} \ ,\end{equation}
 where $S_n\phi(x)$ denotes the ergodic sum $\phi(x)+\cdots+\phi(T^{n-1}x)$. When $\A=\N$, we denote $P_{\N}(T,\phi)$ by $P(T,\phi)$, which is the pressure function that appeared in the introduction.

We will also use the notation  $$
{\text{Var}}_n(\phi):=\sup\big\{|\phi(x)-\phi(y)|:\ I_n(x)=I_n(y)\big\}
$$ for the $n$th variation of $\phi.$

The existence of the limit in the definition of the pressure function \eqref{pre} is guaranteed by the following proposition \cite{M96}.
\begin{pro}[Mauldin-Urba\'nski, 1999]\label{ppp3} Let $\phi: [0,1)\rightarrow\R$ be a real function with ${\text{Var}}_1(\phi)<\infty$ and ${\text{Var}}_n(\phi)\rightarrow 0$
as $n\rightarrow \infty$. Then the limit defining $P_A(T,\phi)$ exists, and the value of $P_A(T,\phi)$ remains the same even without taking supremum over $x\in X_A$ in \eqref{pre}.
\end{pro}

Henceforth, without causing any confusion, when we need to take a point $y$ from a cylinder $I_n(a_1,\ldots,a_n),$ we always
take it as $y=p_n/q_n=[a_1,\ldots,a_n].$ Because all the potentials in the sequel satisfy the condition in Proposition \ref{pp3}, the pressure function can be expressed as
$$P_\A(T, \phi)=\lim_{n\to \infty}\frac{1}{n}\log \sum_{(a_1,\ldots,a_n)\in \A^n}e^{S_n\phi([a_1,\ldots,a_n])} .$$

The following proposition states that in the system of continued fractions the pressure function has a continuity property when the system $\big([0,1),T)$ is approximated by its sub-systems $(X_{A},T).$ For the proof, see \cite{H02} or \cite{LWWX}.
\begin{pro}[Hanus-Mauldin-Urba\'nski, 2002]\label{pp4}~Let $\phi:[0,1)\rightarrow \mathbb{R}$ be a real function with ${\text{Var}}_1(\phi)<\infty$ and ${\text{Var}}_n(\phi)\rightarrow 0$ as $n\rightarrow \infty.$ Then
\begin{enumerate}
  \item[\rm (1)] for any $a\in \mathbb{R}$ and $\A\subset \mathbb{N}, P_{\A}(T,\phi+a)=P_{\A}(T,\phi)+a;$
  \smallskip
  
  \item[\rm (2)]  $P(T,\phi)=P_{\mathbb{N}}(T,\phi)=\sup\{P_{\A}(T,\phi): \A ~\text{is a finite subset of}~ \mathbb{N}\}.$
\end{enumerate}
\end{pro}

Now we specify the potential $\phi$ which will be related to the dimension of the set ${\D_{\mathbf t}}(\Psi)$ when $\Psi(n)=B^n$ for all $n\ge 1$.

Let the function $f_{{t_0,t_1}}$ be as in \eqref{ft1t2}.
Then for any $s\ge 0$, take the potential as $$
\psi(x)=-s\log |T'(x)|-f_{{t_0,t_{1}}}(s)\log B.
$$

For any subset $\A\subset \N,$ ~define
\begin{align*}
 & s^{(2)}(\A, B)=\inf\Big\{s\geq 0:P_{\A}(T,-s\log |T'(x)|-f_{{t_0, t_{1}}}(s)\log B)\leq 0\Big\}, \label{pre4}\\
&s^{(2)}_{n}(\A, B)=\inf\Big\{s\geq 0:\sum\limits_{a_1,\ldots,a_{n}\in \A}
\left(\frac{1}{B^{nf_{{t_0,t_{1}}}(s)}}\right)\left(\frac{1}{q_n^{2}(y)}\right)^s\leq 1\Big\},
\end{align*}
where $y\in I_n(a_1,\ldots,a_n).$  If $\A$ is a finite subset of $\N$, when substitute $s$ by $s^{(2)}(\A, B)$ in the pressure function $P_{\A}$ above (or respectively $s^{(2)}_{n}(\A, B)$ in the summation), we will get an equality.

For simplicity, \begin{itemize}\item when $\A=\N$, write $s^{(2)}(B)$ for $s^{(2)}(\N, B)$ and $s^{(2)}_{n}(B)$ for $s^{(2)}_{n}(\N, B)$;

\item when $\A=\{1, 2, \dots, M\}$ for some integer $M\ge 1$, write them as ~$s^{(2)}(M, B)$ and $s^{(2)}_{n}(M, B)$ respectively.
\end{itemize}

Applying Proposition \ref{pp4}(2) to the potential $\psi,$ one has
\begin{cor}\label{cor5}
$$s^{(2)}(B)=s^{(2)}(\N, B)=\sup\{s^{(2)}(\A, B): \A ~\text{is a finite subset of}~ \mathbb{N}\}.$$
\end{cor}

Then it follows from the definition of pressure function and Corollary \ref{cor5} that

\begin{pro}\label{pp6}~For any $M\in \mathbb{N},$~we have
$$\lim\limits_{n\rightarrow \infty}s^{(2)}_{n}(M, B)=s^{(2)}(M, B),~~\lim\limits_{n\rightarrow \infty}s^{(2)}_{n}(B)=s^{(2)}(B),
~~\lim\limits_{M\rightarrow \infty}s^{(2)}(M, B)=s^{(2)}(B).$$

\end{pro}
%
\begin{pro} \label{tb} As a function of $B\in (1, \infty)$, $s^{(2)}(B)$ is continuous and
$$\lim_{B\to 1}s^{(2)}(B)=1, \quad \lim_{B\to \infty}s^{(2)}(B)=\frac12.$$
\end{pro}
\begin{proof}The proof follows similarly to \cite{WaWu08} without much difference.
\end{proof}

\section{Proof of Theorem \ref{mainLeb}}

We first recall a dynamical Borel-Cantelli lemma from the paper of Kleinbock-Wadleigh \cite[Lemma 3.5]{KleinbockWadleigh}, which is essentially taken from the work of Philipp \cite{Philipp} and follows from the effective mixing property of $T$.

\begin{lem}\label{KWlemma}
Fix $k\in \N$. Suppose $\{A_n: n\ge 1\}$ is  a sequence of sets such that for each $n\ge 1$, the set $A_n$ is a {countable} union of sets of form $$
E_{\bold{r}}=\left\{x\in [0, 1]\smallsetminus \mathbb Q: a_1(x)=r_1, \ldots, a_k(x)=r_k\right\}.$$
Then $T^nx\in A_n$ for infinitely many $n\in \N$  for almost all $x$ or almost no $x$ depending upon the divergence or convergence of the series ${\sum_{n= 1}^\infty}\mu_G(A_n)$ respectively.
\end{lem}

For each $n\geq 1$ and fixed $m\geq 1$,  define \begin{equation}\label{an}
A_n=\left\{x\in [0,1): \prod_{i=1}^ma_{i}^{t_{{i-1}}}(x)\ge \Psi(n)\right\}.
\end{equation}
The set $A_n$ can further be written as the union over a collection of $m$-th order cylinders as$$
A_n=\bigcup_{(a_1,\ldots, a_{m})\in \N^{m}:  \ {a_1^{t_0}\cdots a_{m}^{t_{m-1}}}\ \geq \Psi(n)}I_m(a_1, \ldots, a_m).
$$
To apply Lemma \ref{KWlemma}, we need only to estimate the Lebesgue measure $\mathcal{L}$ of $A_n$, which is equivalent to its Gauss measure $\mu_G$. It follows from Proposition \ref{pp3} that
$$\mathcal L (A_n)\asymp \sum_{{a_1^{t_0}\cdots a_{m}^{t_{m-1}}}\ge \Psi(n)}\prod_{i=1}^m\frac{1}{a_i(a_i+1)},$$
where the constant involved in $\asymp$ depends only on $m$.%

\begin{lem}\label{LebesgueAn}
  Let ${t_0,  \dots, t_{m-1} }$ be an $m$-tuple of positive real numbers, and define
$$
t_{\max}=\max\{t_i: {0\le i\le m-1}\}, \ \ell=\# \{i: t_i=t_{\max}\}.
$$ Then for any $m\ge 1$ and $g\ge 1$, we have \begin{equation}\label{f3}
\sum_{{a_1^{t_0}\cdots a_{m}^{t_{m-1}}}\ge g}\ \prod_{i=1}^m\frac{1}{a_i(a_i+1)}\asymp \frac{(\log g)^{\ell-1}}{g^{\frac1{t_{\max}}}},
\end{equation}
where the constant implied in $\asymp$ depends on $m$ but not on $g$.
\end{lem}


\begin{proof} The summation in (\ref{f3}) does not depend upon the ordering of the partial quotients, therefore without loss of generality we assume that ${t_0\ge  \cdots \ge t_{m-1}} $ and then
$$t_{\max}=t_{{0}}, \ {\text{and}}\ \ \ell=\# \{i: t_i=t_{{0}}\}.$$

We prove this lemma by induction on $\ell\ge 1$.

\begin{itemize}

\item[(I)] When $\ell=1$, we show that (\ref{f3}) holds for all $m\ge 1$. Write $d=m-\ell$. Then it suffices to show (\ref{f3}) holds for all $d\ge 0$. This is done by induction on $d$.

\begin{itemize}
\item [(Ia)] When $d=0$, i.e. $m=1$, it is {easy} to see that (\ref{f3}) holds.

\item[(Ib)] Assume that the result holds for $d-1$; we show that (\ref{f3}) still holds for $d$. Notice that
\begin{align*}
 \sum_{{a_1^{t_0}\cdots a_{m}^{t_{m-1}}}\ge g}\ \prod_{i=1}^m\frac{1}{a_i(a_i+1)}
&\asymp \sum_{a_m^{t_{{m-1}}}\ge g}\frac{1}{a_m(a_m+1)}+\sum_{1\le a_m^{t_{{m-1}}}\le g}\frac{1}{a_m^2}\sum_{a_1^{t_0}\cdots a_{m-1}^{t_{{m-2}}}\ge g/a_m^{t_{{m-1}}}}\prod_{i=1}^{m-1}\frac{1}{a_i(a_i+1)}\\
& \asymp g^{-\frac{1}{t_{{m-1}}}}+\sum_{1\le a_m^{t_{{m-1}}}\le g}\frac{1}{a_m^2}\cdot \left(\frac{a_m^{t_{{m-1}}}}{g}\right)^{\frac{1}{t_{{0}}}}\  {\text{(by induction on {the} inner summation)}}\\
&\asymp g^{-\frac{1}{t_{{m-1}}}}+g^{-\frac{1}{t_{{0}}}}\asymp g^{-\frac{1}{t_{{0}}}},
\end{align*}
where the second last quantity is obtained by noticing that ${t_{m-1}/t_0}<1$, so the summation over $a_m$ converges.
\end{itemize}
\medskip

\item[(II)] Assume that $\ell\ge 2$. As for (I) above, we use induction on $d=m-\ell$.
\begin{itemize}
\item [(IIa)]When $d=0$, i.e. $m=\ell$ and $t_i=t$ for all ${0\le i\le m-1}$, {we have}
 \begin{align*}
 &\sum_{{a_1^{t_0}\cdots a_{m}^{t_{m-1}}}\ge g}\ \prod_{i=1}^m\frac{1}{a_i(a_i+1)}\\
&\asymp \sum_{a_m^{t}\ge g}\frac{1}{a_m(a_m+1)}+\sum_{1\le a_m^{t}\le g}\frac{1}{a_m^2}\sum_{a_1^{t}\cdots a_{m-1}^{t}\ge g/a_m^{t}}\prod_{i=1}^{m-1}\frac{1}{a_i(a_i+1)}\\
& \asymp g^{-1/t}+\sum_{1\le a_m^{t}\le g}\frac{1}{a_m^2}\cdot
\frac{(\log \frac{g}{a_m^t})^{\ell-2}}{(\frac{g}{a_m^t})^{1/t}},\ \ \ \ \ \ \ {\text{(by induction on inner summation)}}\\
&\asymp \frac{1}{g^{1/t}}+\int_1^{g^{1/t}}
\frac{1}{x^2}
\cdot
\ \frac{(\log \frac{g^{1/t}}{x})^{\ell-2}}{(\frac{g^{1/t}}{x})}\,dx,\ \   {\text{(change variable $y=\frac{g^{1/t}}{x}$)}} \\ &\asymp \frac{1}{g^{1/t}}+\int_1^{g^{1/t}}\frac{1}{g^{1/t}}\cdot
\frac{(\log y)^{\ell-2}}{y}\,dy
\asymp \frac{1}{g^{1/t}}+\frac{(\log g)^{\ell-1}}{g^{1/t}}.
\end{align*}

\item [II(b)] Assume that the result holds for $d-1$. We show that (\ref{f3}) still holds for any $d$. Since $\ell$ is fixed, it means that $$
{t_0= \cdots=t_{\ell-1}>t_{\ell}\ge \cdots \ge t_{m-1}}.
$$ So, \begin{equation*}\label{f4}
\#\{{i\ge 1: t_i=t_1\}=\ell-1, \ {\text{and}}\ t_0=t_1}.
\end{equation*}Notice that \begin{align*}
I:= &\sum_{{a_1^{t_0}\cdots a_{m}^{t_{m-1}}}\ge g}\ \prod_{i=1}^m\frac{1}{a_i(a_i+1)}\\ &\asymp { \sum_{a_1^{t_0}\ge g}\frac{1}{a_1(a_1+1)}+\sum_{1\le a_1^{t_0}\le g}\frac{1}{a_1^2}\sum_{a_2^{t_1}\cdots a_{m}^{t_{m-1}}\ge g/a_1^{t_0}}\ \prod_{i=2}^{m}\frac{1}{a_i(a_i+1)}}.\end{align*}
For the inner summation, the induction hypothesis is applied to give
{\begin{align*}
I &   \asymp \frac{1}{g^{\frac{1}{t_0}}}+\sum_{1\le a_1^{t_0}\le g}
\frac{1}{a_1^2}\cdot \frac{\left(\log \frac{g}{a_1^{t_0}}\right)^{\ell-2}}{\left(\frac{g}{a_1^{t_0}}\right)^{1/t_1}}
\\
&  \asymp  \frac{1}{g^{\frac{1}{t_0}}}+\sum_{1\le a_1^{t_0}\le g}\frac{1}{a_1^2}\cdot \frac{\left(\log \frac{g^{1/t_0}}{a_1}\right)^{\ell-2}}{\left(\frac{g}{a_1^{t_0}}\right)^{1/t_0}}\ \ \ \ {\text{(by $t_0=t_1$)}}.
\end{align*}}
So we get the same formula as in case (IIa).
\end{itemize}
\end{itemize}
\end{proof}

{Now observe that $${\D_{\mathbf t}} (\Psi) = \left\{x\in[0,1): T^{n-1}x\in A_n\text{ for infinitely many }n\in \N\right\},$$
where $A_n$ are as in \eqref{an}.}
By combining
Lemmas \ref{KWlemma} and \ref{LebesgueAn},  we conclude that $\mathcal L \big({\D_{\mathbf t}} (\Psi)\big)$ is zero or
full according to the convergence or divergence of the series 
$$\sum\limits_{n}^\infty\mathcal L(A_n)\asymp\sum\limits_{n}^\infty\frac {\big(\log\Psi(n)\big)^{\ell-1}}{\Psi(n)^{1/t_{\max}}},$$
where $
t_{\max}$ and $ \ell$ are as in \eqref{tmax}. {This finishes the proof of Theorem \ref{mainLeb}.}
\section{Hausdorff dimension for $B=1$ or $B=\infty$
}  {In this section we} prove Theorem \ref{BHWthmeasy}
\ignore{ $$
\log B=\liminf_{n\to \infty}\frac{\log \Psi(n)}{n}, \ \ \log b=\liminf_{n\to\infty}\frac{\log \log \Psi(n)}{n}.
$$
Then various cases of  growth rate of $\Psi(n)$ comes into play. These include}
by considering the 
{two cases:
\begin{itemize}
\item $B=1$;
\item $B=\infty$  (for this case, there are three subcases $b=1,1<b<\infty,$ and $b=\infty$).
\end{itemize}
\ignore{We will determine the Hausdorff dimension of the set
 $$
{\D_{\mathbf t}} (B):=\left\{x\in[0, 1):  a_{n+1}^{t_1}(x)a_{n+2}^{t_2}(x) \ge B^n \ {\text{for i.m.}} \ n\in \N\right\}.
$$
by discussing all the cases one by one in different subsections below.} We start off with the easier case.}

\subsection{$B=1$} It is trivial that
$$
{\D_{\mathbf t}}(\Psi)\supset \left\{x\in[0, 1):  {a^{t_0}_{n}}(x)\ge \Psi(n)\  {\text{for infinitely many}} \ n\in \N       \right\}.$$
It follows from Theorem \ref{WaWu} that the set on the right hand side has full Hausdorff dimension. Hence $\hdim {\D_{\mathbf t}}(\Psi)=1$ when $B=1$.

\subsection{ $B=\infty$} There are three subcases.

\subsubsection{ $1<b<\infty$}\

By the definition of $b$, for any $c<b$, $$\frac{\log\log \Psi(n)}{n}\geq \log c, \ \text{i.e.} \ \Psi(n)\geq e^{c^n}$$
for all sufficiently large $n$ which we write as $n\gg1$. Thus for any $x\in {\D_{\mathbf t}}(\Psi)$,  there are infinitely many $n$ such that $$
 {\prod_{i=0}^{m-1}a^{t_i}_{n+i}(x)}\ge e^{c^n},$$
then at least for one index ${0\leq i\leq m-1}$ one has $a^{t_i}_{n+i}(x)\ge e^{\frac{1}{{m}}\cdot c^n}.$
Thus $$
{\D_{\mathbf t}}(\Psi)\subset { \bigcup_{i=0}^{m-1} \left\{x\in[0, 1):  a^{t_i}_{n+i}(x)\ge e^{\frac{1}{m}\cdot c^n}\  {\text{for i.m.}} \ n\in \N  \right\}}
$$
It follows from Lemma \ref{lemb} that {each of} the sets on the right hand side have Hausdorff dimension $\frac{1}{1+c}$ irrespective of $t_i's$. Hence $\hdim {\D_{\mathbf t}}(\Psi)\leq \frac{1}{1+b}$ by the arbitrariness of $c<b$.

On the other hand, by the definition of $b$, it follows that for any $d>b$, $$
\Psi(n)\le e^{d^n}, \ {\text{for infinitely many}}\ n\in \N.
$$ Thus one has
$$
{\D_{\mathbf t}}(\Psi)\supset \left\{x\in[0, 1):  {a^{t_{0}}_{n}}\ge e^{d^n}\  {\text{for all}} \ n\in \N       \right\},
$$ and   from Lemma \ref{lemb} we {conclude that} the Hausdorff dimension of the set on the right hand side is  $1/(1+d)$.

\subsubsection{$b=\infty$} This case readily follows from the upper bound argument above, that is,
$${\D_{\mathbf t}}(\Psi)\leq \lim_{b\to\infty}\frac1{b+1}= 0.$$

\subsubsection{$b=1$}
In this case, for any $\epsilon>0$, $\ \Psi(n)\leq e^{(1+\epsilon)^n}$ for infinitely many $n$. Then
 \begin{align*}
{\D_{\mathbf t}}(\Psi)&\supset  \left\{x\in[0, 1):  {a^{t_0}_{n}}\ge \Psi(n)\  {\text{for infinitely many}} \ n\in \N       \right\}\\ & \supset
\left\{x\in[0, 1):  {a^{t_0}_{n}}(x)\ge e^{(1+\epsilon)^n}\  {\text{for all }} \ n\in \N       \right\} .\end{align*}
Hence by using Lemma \ref{lemb}, we have
$$\hdim {\D_{\mathbf t}}(\Psi)\geq \lim_{\epsilon\to 0}\frac{1}{1+1+\epsilon}=\frac12.$$

For the upper bound, we note that $$
 {\prod_{i=0}^{m-1}a^{t_i}_{n+i}(x)}\ge \Psi(n) \Longrightarrow {a^{t_i}_{n+i}\geq \Psi(n)^{\frac1{m}}\quad \text{for some}  \ 0\leq i\leq m-1}.$$
Hence
$${\D_{\mathbf t}}(\Psi)\subseteq {\bigcup_{i=0}^{m-1}\left\{x\in [0, 1): a^{t_i}_{n+i}\geq \Psi(n)^{\frac1{m}}, \ \text{for i.m.} \ n\in \mathbb N\right\}}
.$$
Since $B=\infty$, for any ${A}>1$ one has
$${\D_{\mathbf t}}(\Psi)\subseteq \left\{x\in [0, 1): a_{n}\geq  {A}^n, \ \text{for i.m.} \ n\in \mathbb N\right\}.$$
Hence by letting ${A}\to\infty$ and appealing to Proposition \ref{tb}, it follows that $\hdim {\D_{\mathbf t}}(\Psi)\leq 1/2$.

\section{$\hdim {\D_{\mathbf t}}(B)$ for {$m=2$ and} $1<B<\infty$: an upper bound}


In the next three sections we specialize to the case $m=2$, that is, take ${\mathbf t = (t_0,t_1)}$, and assume that $1 < B < \infty$. To prove 
Theorem \ref{BHWthm}, we first show that the Hausdorff dimension of the set
\begin{equation}
\label{d2b}
{\D_{\mathbf t}}(B):=\left\{x\in[0, 1):  a_{n}^{t_0}(x)a_{n+1}^{t_1}(x) \ge B^n \ {\text{for i.m.}} \ n\in \N\right\}
\end{equation}
is equal to  $$\inf \left\{s\ge 0: P\big(T, -s\log |T'|-f_{t_0, t_{1}}(s)\log B\right)\le 0\big\},$$
where $f_{t_0, t_{1}}$ is as in \eqref{ft1t2}.

We recall that according to Theorem \ref{WaWu}, 
the Hausdorff dimension of the one-parameter version of  \eqref{d2b}, namely the set
$$
\left\{x\in[0, 1):   a_{n}(x)^{t_0}\ge B^n \ {\text{for i.m.}} \ n\in \N
\right\},
$$
is given by
$$
\inf \left\{s\ge 0: P\left(T, -s\log |T'|-\frac{s}{t_0}\log B\right)\le 0\right\}.
$$
This gives the first function $f_{t_0}$ defined by $$
f_{t_0}(s)=\frac{s}{t_0}.
$$

%
%
{Now take a positive number $A$ with $1<A<B$ and} define
\begin{align*}
 {\D'_{\mathbf t}}({A})&:= \left\{x\in[0, 1):   {a_{n}^{t_0}(x) \le A^n, \  a_{n+1}^{t_1}(x)\ge \frac{B^n}{a_{n}^{t_0}(x)}}
 \  {\text{for i.m.}} \ n\in \N       \right\},\end{align*}
  and
 \begin{align*}
{\D''_{\mathbf t}}({A})&:=\left\{x\in[0, 1):  {a_{n}^{t_0}}(x) \ge A^n\  {\text{for i.m.}} \ n\in \N       \right\}.\end{align*} Then \begin{equation*}\label{equpper}
{\D_{\mathbf t}}(B)\,\subset\, {\D'_{\mathbf t}}({A})\cup {\D''_{\mathbf t}}({A}),
\end{equation*}

From {the} $m=1$ case above,  the Hausdorff dimension of the set ${\D''_{\mathbf t}}({A})$ is given by \begin{equation}\label{f1}
\hdim {\D''_{\mathbf t}}({A})=\inf\left\{s\ge 0: P\left(T, -s\log |T'|-f_{t_{{0}}}(s)\log A\right)\le 0\right\}:=\delta_1.
\end{equation}

Now we focus on the Hausdorff dimension of 
${\D'_{\mathbf t}}({A})$.
{Since it} readily follows from Theorem \ref{WaWu} that $ 1/2<\hdim{\D_{\mathbf t}}(B)<1$ for $1<B<\infty$,
we consider the $s$-Hausdorff measure of ${\D'_{\mathbf t}}({A})$ only for $1/2<s<1$.

Because of the limsup nature of ${\D'_{\mathbf t}}(A)$, there is a natural cover of it. For any integers $a_1,\dots, a_{n}$, define
 $$
J_{n}(a_1,\dots, a_{n}) :=\bigcup_{a_{n+1}: \  a_{n+1}^{t_{{1}}}\ge \frac{B^n}{a_{n}^{t_{{0}}}}}I_{n+1}(a_1,\dots, a_{n+1}).
$$ Then $$
{\D'_{\mathbf t}}({A})=\bigcap_{N=1}^{\infty}\bigcup_{n=N}^{\infty}\bigcup_{a_1,\dots, a_{n-1}\in \N}\ \bigcup_{a_{n}^{t_{{0}}}\le A^n}J_{n}(a_1,\dots, a_{n}).
$$

By Proposition \ref{pp3}, one has \begin{align*}
|J_{n}(a_1,\dots, a_{n})|\asymp & \left[{q_{n-1}^2a_{n}^2\left(\frac{B^n}{a_{n}^{{{t_0}}}}\right)^{\frac{1}{t_1}}}\right]^{-1}=\left[{q_{n-1}^2a_{n}^{2-\frac{{{t_0}}}{{{t_1}}}} B^{\frac{n}{t_{1}}}}\right]^{-1},
\end{align*}  where the constant implied in $\asymp$ can be chosen as $4$.

Thus  the $s$-Hausdorff measure of ${\D'_{\mathbf t}}({A})$  can be estimated as
\begin{align*}
\mathcal{H}^s\big({\D'_{\mathbf t}}({A})\big)&\le \liminf_{N\to \infty}\sum_{n=N}^{\infty}\sum_{a_1,\dots, a_{n-1}\in \N}\ \sum_{a_{n}^{{{t_0}}}\le A^n} |J_{n}(a_1,\dots,a_{n})|^s\\
&\ll  \liminf_{N\to \infty}\sum_{n=N}^{\infty}\sum_{a_1,\dots, a_{n-1}\in \N}\ \sum_{a_{n}^{{{t_0}}}\le A^n}
\left[{q_{n-1}^2a_{n}^{2-\frac{{{t_0}}}{{{t_1}}}} B^{\frac{n}{t_{1}}}}\right]^{-s}\\
& \asymp \liminf_{N\to \infty}\sum_{n=N}^{\infty}\sum_{a_1,\dots, a_{n-1}\in \N}\
 \sum_{a_{n}^{{{t_0}}}\le A^n}a_{n}^{-(2-\frac{{{t_0}}}{{{t_1}}})s} 
\left({q_{n-1}^2B^{\frac{n}{{{t_1}}}}}\right)^{-s}.
\end{align*}
Calculating the summation over $a_{n}$ gives that \begin{align*}
\sum_{a_{n}^{{{t_0}}}\le A^n}a_{n}^{-\big(2-\frac{{{t_0}}}{{{t_1}}}\big)s} \ll \max\Big\{1, A^{\frac{n}{{{t_0}}}\cdot \big(1-s(2-\frac{{{t_0}}}{{{t_1}}})\big)}\Big\}=\max\Big\{1, A^{n \big(\frac{1-2s}{{{t_0}}}+\frac{s}{{{t_1}}}\big)}\Big\}.
\end{align*}
Thus $$
\mathcal{H}^s\big({\D'_{\mathbf t}}({A})\big)\le \liminf_{N\to \infty}\sum_{n=N}^{\infty}\sum_{a_1,\dots, a_{n-1}\in \N}\max\Big\{1, A^{n \big(\frac{1-2s}{{{t_0}}}+\frac{s}{{{t_1}}}\big)}\Big\}\cdot \left({q_{n-1}^2B^{\frac{n}{{{t_1}}}}}\right)^{-s}.
$$
This gives an upper bound of the Hausdorff dimension of the set ${\D'_{\mathbf t}}({A})$ to be
\begin{equation}\label{f2}
\inf\left\{s\ge 0: P\bigg(T, -s\log |T'|+\max\left\{0, \frac{1-2s}{{{t_0}}}+\frac{s}{{{t_1}}}\right\}\log A-\frac{s}{t_{2}}\log B\bigg)\le 0\right\}:=\delta_2.\end{equation}

Combining (\ref{f1}) and (\ref{f2}), one gets $$
 \hdim {\D_{\mathbf t}}(B)\le \max\{\delta_1,\delta_2\}.
$$ It would be reasonable to choose $A$ such that $\delta_1=\delta_2$ which would give the optimal upper bound of $\hdim {\D_{\mathbf t}}(B)$.
Choose $A$ such that the potentials in $\delta_1$ and $\delta_2$ are equal, namely, $$
-f_{{{t_0}}}(s)\log A=\max\left\{0, \frac{1-2s}{{{t_0}}}+\frac{s}{{{t_1}}}\right\}\log A-\frac{s}{{{t_1}}}\log B
$$ equivalently \begin{equation*}\label{f6}
\log A=\frac{s}{{{t_1}}f_{{{t_0}}}(s)+\max\left\{0, s-\left(2s-1\right)\frac{{{t_1}}}{{{t_0}}} \right\}}\log B.
\end{equation*}
Then define $f_{{t_0, t_{1}}}$ such that $$
-f_{{t_0,t_{1}}}(s)\cdot \log B=-f_{{{t_0}}}(s)\cdot \log A
$$ giving that (note $s>1/2$)\begin{align}\label{f5}
f_{{t_0, t_{1}}}(s)&=\frac{sf_{{{t_0}}} (s)}{{{t_1}}f_{{{t_0}}}(s)+\max\left\{0, s-(2s-1)\frac{{{t_1}}}{{{t_0}}}\right\}}=\frac{sf_{{{t_0}}}(s)}{{{t_1}}\left[f_{{{t_0}}}(s)+\max\{0, \frac{s}{{{t_1}}}-\frac{2s-1}{{{t_0}}}\}\right]}.
\end{align}
Note that   \eqref{f5} is the same as \eqref{ft1t2} given in the statement of the Theorem \ref{BHWthm}. 
As a result, once we can check that the chosen $A$ is less than $B$,  we {will}
arrive at the final conclusion $$
\hdim {\D_{\mathbf t}}(B) \le \inf \{s\ge 0: P(T, -s\log |T'|-f_{{t_0, t_{1}}}(s)\log B)\le 0\}.
$$

We show that $A<B$ in the following lemma.
\begin{lem}\label{lemrec} For any $0<s<1$, $$f_{{t_0, t_{1}}}(s)< f_{{{t_0}}}(s), \ \ {\text{or, equivalently,}} \ A<B.$$
\end{lem}
\begin{proof}Recall (\ref{f5}). Then \begin{align*}
{f_{{t_0,{{t_1}}}}}(s)< f_{{{t_0}}}(s)%
&\Longleftrightarrow s<{{t_1}}f_{{{t_0}}}(s)+\max\left\{0, s-(2s-1)\frac{{{t_1}}}{{{t_0}}}\right\}\\
&\Longleftarrow s<{{t_1}}f_{{{t_0}}}(s)+s-(2s-1)\frac{{{t_1}}}{{{t_0}}}\\
&\Longleftrightarrow (2s-1)\frac{{{t_1}}}{{{t_0}}}<{{t_1}}f_{{{t_0}}}(s)={{t_1}}\cdot \frac{s}{{{t_0}}}.
\end{align*}
The last estimate is nothing but to say $2s-1<s$, which is clearly true since $s<1$.

\end{proof}

\section{$\hdim {\D_{\mathbf t}}(B)$ for $1<B<\infty$: {a}  lower bound}\label{s6}

To obtain the lower bound, we will construct an appropriate Cantor subset of ${\D_{\mathbf t}}(B)$ and then apply the following mass distribution principle \cite{Falconer_book}.
\begin{pro}[Mass Distribution Principle]\label{p1}
Let $\mu$ be a probability measure supported on a measurable set $F$. Suppose there are positive constants $c$ and $r_0$ such that
$$\mu\big(B(x,r)\big)\le c r^s$$
for any ball $B(x,r)$ with radius $r\le r_0$ and center $x\in F$. Then $\hdim F\ge s$.
\end{pro}

\subsection{Preliminaries on the dimension estimate}\
Recall that $$
f_{{{t_0}}}(s)=\frac{s}{{{t_0}}}, \ \ f_{{t_0, t_1}}(s)=\frac{sf_{{{t_0}}}(s)}{{{t_1}}\left[f_{{{t_0}}}(s)+\max\{0, \frac{s}{{{t_1}}}-\frac{2s-1}{{{t_0}}}\}\right]},
$$
and {write $s_o$ for $s^{(2)}(B)$, i.e.}$$
s_o=\inf\left\{s\ge 0: P(-s\log |T'|-f_{{t_0,t_1}}(s)\log B)\le 1\right\}.
$$

We present some facts about this dimension estimate. The following may be trivial, however we give a rigorous proof to avoid any potential uncertainty.
Define $$
s_o'=\Big\{s\ge 0: P(T, -s\log |T'|-\frac{s}{{{t_1}}}\log B)\le 0\Big\}.
$$
\begin{lem}\label{l6.2} When $\frac{s_o}{{{t_1}}}-\frac{2s_o-1}{{{t_0}}}\le 0$, one has $s_o=s_o'$.\end{lem}
\begin{proof}

At first, remember that the pressure function $P(T, \cdot)$ is non-decreasing with respect to the potential, i.e.
$$P(T, \psi_1)\le P(T, \psi_2), \ \ {\text{if}}\ \psi_1\le \psi_2.$$  Note that we always have $$
f_{{t_0,t_1}}(s)\le \frac{sf_{{{t_0}}}(s)}{{{t_1}}[f_{{{t_0}}}(s)+0]}=\frac{s}{{{t_1}}}.
$$ Thus $$
-s\log |T'|-f_{{t_0,t_1}}(s)\log B\ge -s\log |T'|-\frac{s}{{{t_1}}}\log B,
$$   which implies that $$
s_o'\le s_o.
$$

For the other direction of the inequality,  we distinguish two cases.
\begin{itemize}
  \item When $\frac{s_o}{{{t_1}}}-\frac{2s_o-1}{{{t_0}}}<0$. Let $\epsilon>0$ be small such that for any $s_o-\epsilon<s<s_o+\epsilon$, we always have
  \begin{equation*}\label{ff3}
  \frac{s}{{{t_1}}}-\frac{2s-1}{{{t_0}}}<0,
  \end{equation*} and so \begin{equation}\label{ff4}
  f_{{t_0, t_1}}(s)=\frac{s}{{{t_1}}}.
  \end{equation}

 For any $s_o-\epsilon<s<s_o$, by the definition of $s_o$ we have $$
  P\big(T, -s\log |T'|-f_{{t_0,t_1}}(s)\log B\big)>0.
  $$ so by (\ref{ff4}), it follows {that} $$
   P\left(T, -s\log |T'|-\frac{s}{{{t_1}}}\log B\right)>0.
  $$
  This implies $s_o'\ge s$. By the arbitrariness of $s$, one has $s_o'\ge s_o$.

\item When $\frac{s_o}{{{t_1}}}-\frac{2s_o-1}{{{t_0}}}=0$.  In this case, one has $$f_{{t_0, t_1}}(s_o)=\frac{s_o}{{{t_1}}}.$$ By the continuity of $f_{{t_0,t_1}}$ with respect to $s$, for any $\epsilon>0$, choose $0<\delta\le \epsilon$, such that for any $s_o-\delta<s<s_o$, $$
    f_{{t_0,t_1}}(s)>\frac{s_o-\epsilon}{{{t_1}}}.
    $$

 On one hand, by the definition of $s_o$, 
 for any $s_o-\delta<s<s_o$, $$
 P(T, -s\log |T'|-f_{{t_0,t_1}}(s)\log B)>0,
 $$ on the other hand, (since $s>s_o-\epsilon$) $$
 -s\log |T'|-f_{{t_0,t_1}}(s)\log B<-(s_o-\epsilon)\log |T'|-\frac{s_o-\epsilon}{{{t_1}}}\log B,
 $$ which implies that $$
0<P\Big(T, -s\log |T'|-f_{{t_0,t_1}}(s)\log B\Big)\le P\Big(T, -(s_o-\epsilon)\log |T'|-\frac{s_o-\epsilon}{{{t_1}}}\log B\Big).
 $$ Thus $s_o'\ge s_o-\epsilon$.
 \end{itemize}

\end{proof}

As a result, when $$
\frac{s_o}{{{t_1}}}-\frac{2s_o-1}{{{t_0}}}\le 0,
$$ we consider the following subset of $\mathcal{E}_{\bold{t}}(B)$: $$
\Big\{x\in [0,1): a_{n+1}^{{{t_1}}}(x)\ge B^n, \ {\text{i.m.}}\ n\in \N\Big\}
$$ which, by Theorem \ref{WaWu}, is of dimension $$
s_o'=\inf\{s\ge 0: P(T, -s\log |T'|-\frac{s}{{{t_1}}}\log B)\le 0\}.
$$ Thus $$
\hdim \mathcal{E}_{\bold{t}}(B)\ge s_o'=s_o.
$$
So in the following, we always assume that $$
\frac{s_o}{{{t_1}}}-\frac{2s_o-1}{{{t_0}}}> 0,
$$ and then in a small neighborhood of $s_o$, we always have \begin{equation}\label{ff5}
f_{{t_0,t_1}}(s)=\frac{sf_{{{t_0}}}(s)}{{{t_1}}\big(f_{{{t_0}}}(s)+\frac{s}{{{t_1}}}-\frac{2s-1}{{{t_0}}}\big)}.
\end{equation}

\subsection{A subset of ${\D_{\mathbf t}}(B)$}\label{s7}\

Fix integers $M,N$ sufficiently large such that $s:={s^{(2)}_N(M,B)}$ is in the small neighborhood of $s_o$ so that $1>s>1/2$ and (\ref{ff5}) holds. Then define a real number $A$ such that
\begin{equation}\label{fff1}
f_{{{t_0}}}(s)\log A=f_{{t_0, t_1}}(s)\log B.
\end{equation}
It is {straightforward} to check that $f_{{{t_0}}}(x)>f_{{t_0, t_1}}(x)$ for any $0<x<1$, so $1<A<B$.

 Fix a sequence of largely sparse integers $\{\ell_k\}_{k\ge 1}$, say, $$
\ell_k\gg e^{\ell_1+\cdots+\ell_{k-1}},\ {\text{and take}}\ n_1=\ell_1 N+1, \ n_{k+1}-n_{k}=\ell_{k+1} N+2, \ \forall\,k\ge 1,
$$ {so that the number of integers in the interval $(n_k+1, n_{k+1})$ is a multiple of $N$.} Then define a subset of ${\D_{\mathbf t}}(B)$ as \begin{align*}
E=\Bigg\{x\in [0,1): A^{\frac{n_k}{{{t_0}}}}\le a_{n_k}(x)&<2 {A^{\frac{n_k}{{{t_0}}}}}, \ \left(\frac{B^{n_k}}{A^{n_k}}\right)^{1/{{t_1}}}\le a_{n_k+1}(x)<2\left(\frac{B^{n_k}}{A^{n_k}}\right)^{1/{{t_1}}} {\text{for all}} \ k\ge 1;\\ &{\text{and}}\ a_n(x)\in \{1,\dots, M\}  \ {\text{for other $n\in \N$}}\Bigg\}.
\end{align*}

For ease of notation, \begin{itemize}
  \item write $$
 \alpha_0=A^{1/{{t_0}}}, \ \  \alpha_1=\left(\frac{B}{A}\right)^{1/{{t_1}}}.
  $$
 \item write $q_n(a_1,\dots, a_n)$ as $q_n$ when the partial quotients  $a_1,\dots, a_n$ are clear. Recall (\ref{fff1}). Then $$
  1=\sum_{1\le a_1,\dots, a_N\le M}\frac{1}{q_N^{2s}(a_1,\dots, a_N)\cdot B^{N\cdot f_{t_0, t_1}(s)}}=\sum_{1\le a_1,\dots, a_N\le M}\frac{1}{q_N^{2s}\cdot \alpha_0^{Ns}}.
  $$
  \item use a symbolic space defined as $D_0=\{\varnothing\}$, and for any $n\ge 1$, \begin{align*}
    D_n=\Bigg\{(a_1,\dots, a_n)\in \N^n: \alpha_i^{n_k}\le a_{n_k+i}&<2\alpha_i^{n_k}  \ {\text{for}} \ 0\le i\le 1, \ k\ge 1 \ {\text{with}} \ {n_k+i}\le n;\\ &{\text{and}}\ a_j\in \{1,\dots, M\}  \ {\text{for other $j\le n$}}\Bigg\},
  \end{align*} which is just the collection of the prefixes of the points in $E$.

  \item if an integer $n$ is assumed as a real value $\xi$, we mean $n=\lfloor \xi\rfloor$ and in $D_n$, the term $a_{n_k+i}$ has $\alpha_i^{n_k}$ choices.

   \item   Use $\U$ to denote the following collection of finite words:
$$
\U=\{w=(\sigma_1,\dots, \sigma_N): 1\le \sigma_i\le M, \ 1\le i\le N\}.
$$ In the following, we always use $w$ to denote a word of length $N$ in $\U$.
\end{itemize}

\subsection{Cantor structure of $E$}\

For any $(a_1,\dots, a_n)\in D_n$, define $$
J_n(a_1,\dots, a_n)=\bigcup_{a_{n+1}: (a_1,\dots, a_n, a_{n+1})\in D_{n+1}}I_{n+1}(a_1,\dots, a_n, a_{n+1})
$$ and call it a {\em basic cylinder} of order $n$. More precisely, for each $k\ge 0$\begin{itemize}

\item when $n_k+1\le n<n_{k+1}-1$ (by viewing $n_0=0$), $$
J_n(a_1,\dots, a_n)=\bigcup_{1\le a_{n+1}\le M}I_{n+1}(a_1,\dots, a_n, a_{n+1})
$$
\item when $n=n_{k+1}-1+i$ for $i=0,1$, $$
J_n(a_1,\dots, a_n)=\bigcup_{\alpha_i^{n_{k+1}}\le a_{n+1}< 2 \alpha_i^{n_{k+1}}}I_{n+1}(a_1,\dots, a_n, a_{n+1})
$$

\end{itemize}Then define $$
\mathcal{F}_n=\bigcup_{(a_1,\dots, a_n)\in D_n}J_n(a_1,\dots, a_n)
$$ and call it level $n$ of the Cantor set $E$. It is clear that $$
E=\bigcap_{n=1}^{\infty}\mathcal{F}_n=\bigcap_{n=1}^{\infty}\bigcup_{(a_1,\dots, a_n)\in D_n}J_n(a_1,\dots, a_n).
$$

We have the following observations about the length and gaps of the basic cylinders.
\begin{lem}[Gap estimation]\label{l3} Denote by $G_n(a_1,\dots, a_n)$ the gap between $J_n(a_1,\dots, a_n)$ and other basic cylinders of order $n$. Then $$
G_n(a_1,\dots, a_n)\ge \frac{1}{M}\cdot |J_n(a_1,\dots, a_n)|.
$$
\end{lem}
\begin{proof}
  This lemma can be observed from the positions of the cylinders in Proposition \ref{pp2}. A detailed proof can be found in \cite{HuWuXu}.
\end{proof}

Recall the definition of $\U$. Every element $x\in E$ can be written as \begin{align*}
x=[w_1^{(1)},\dots, w_{\ell_1}^{(1)}, a_{n_1}, a_{n_1+1}, & w_1^{(2)},\dots, w_{\ell_2}^{(2)}, a_{n_2},a_{n_2+1},\\
 \dots, & w_1^{(k)},\dots, w_{\ell_k}^{(k)}, a_{n_k},a_{n_k+1},\dots],
\end{align*} where $w\in \U$ and
$$\alpha_0^{n_k}\le a_{n_k}< 2\alpha_0^{n_k}, \ \alpha_1^{n_k}\le a_{n_k+1}< 2\alpha_1^{n_k},\ \ {\text{for all}}\ k\ge 1.$$
\begin{lem}[Estimation on $q_n(x)$]\label{l6.4} Let $n_k+1< n\le n_{k+1}+1$. \begin{itemize}
  \item $n=n_k+1+\ell N$ for some $1\le \ell\le \ell_{k+1}$, $$
  q_{n_k+1+\ell N}(x)\le \left(2^{\ell}\cdot \prod_{i=1}^{\ell}q_N(w_i^{(k+1)})\right)\cdot \prod_{t=1}^{k} \left(2^{\ell_t+4}\alpha_0^{n_t} \alpha_1^{n_t}\prod_{l=1}^{\ell_t}q_N(w_l^{(t)})\right).
  $$

  \item $n=n_{k+1}$, $$
  q_{n_{k+1}}(x)\le \left(2^{\ell_{k+1}+2}\alpha_0^{n_{k+1}}\cdot \prod_{i=1}^{\ell_{k+1}}q_N(w_i^{(k+1)})\right)\cdot \prod_{t=1}^{k} \left(2^{\ell_t+4}\alpha_0^{n_t} \alpha_1^{n_t}\prod_{l=1}^{\ell_t}q_N(w_l^{(t)})\right).
  $$
  \item $n=n_{k+1}+1$, $$
  q_{n_{k+1}+1}(x)\le \prod_{t=1}^{k+1} \left(2^{\ell_t+4}\alpha_0^{n_t} \alpha_1^{n_t}\prod_{l=1}^{\ell_t}q_N(w_l^{(t)})\right).
  $$
  \item for any $n$ with $n_k+1+(\ell-1)N<n<n_k+1+\ell N$, $$
  \frac{1}{(M+1)^N}\cdot q_{n_k+1+\ell N}(x)\le q_n(x)\le (M+1)^{N}\cdot q_{n_k+1+(\ell-1)N}(x).
  $$
\end{itemize}
\end{lem}
\begin{proof}
  Use the second item in Proposition \ref{pp3} recursively to get the first estimation. More precisely, \begin{align*}
    q_{n_k+1+\ell N}(x)&\le \left(2^{\ell}\cdot \prod_{i=1}^{\ell}q_N(w_i^{(k+1)})\right)\cdot q_{n_k+1}(x)\\
    &\le \left(2^{\ell}\cdot \prod_{i=1}^{\ell}q_N(w_i^{(k+1)})\right)\cdot \left(2^{\ell_k+4}\alpha_0^{n_k} \alpha_1^{n_k}\prod_{l=1}^{\ell_k}q_N(w_l^{(k)})\right)\cdot q_{n_{k-1}+1}(x).
  \end{align*}For the next two, one just use $$
  q_{n+1}=a_{n+1}q_n+q_{n-1}\le (a_{n+1}+1)q_n.
  $$ For the last item, note that the partial quotients $$
  1\le a_n\le M, \ {\text{for all}}\ n_k+1+(\ell-1)N<n<n_k+1+\ell N.
  $$
\end{proof}

We estimate the length of basic cylinders $J_n(x)$ for all $n\ge 1$. For $n_k+1\le n<n_{k+1}-1$, we have
\begin{align*}
  |J_n(x)|=\left|\frac{p_n+p_{n-1}}{q_n+q_{n-1}}-\frac{(M+1)p_n+p_{n-1}}{(M+1)q_n+q_{n-1}}\right|=\frac{M}{(q_n+q_{n-1})((M+1)q_n+q_{n-1})}\ge \frac{1}{8q_n^2},
\end{align*} and similarly, $$
\frac{1}{\alpha_0^{n_{k}}q_{n_{k}-1}^2(x)}>|J_{n_{k}-1}(x)|\ge \frac{1}{8\alpha_0^{n_{k}}q_{n_{k}-1}^2(x)},\ \ \ \
\frac{1}{\alpha_1^{n_{k}}q_{n_{k}}^2(x)}>|J_{n_{k}}(x)|\ge \frac{1}{8\alpha_1^{n_{k}}q_{n_{k}}^2(x)}.
$$ Consequently, we have 
\begin{lem}[Length estimation]\label{l6.5} Let $n_k-1\le n<n_{k+1}-1$.

\begin{itemize}
\item For $n=n_{k}-1=n_{k-1}+1+\ell_k N$,
\begin{equation*}\label{ff8}|J_{n_{k}-1}(x)|\ge \frac{1}{2^3\alpha_0^{n_{k}}}\cdot \left(\frac{1}{2^{2\ell_{k}}}\cdot \prod_{i=1}^{\ell_{k}}\frac{1}{q_N^2(w_i^{(k)})}\right)\cdot \left[\prod_{t=1}^{k-1}\left(\frac{1}{2^{2\ell_t+8}}\cdot\frac{1}{\alpha_0^{2n_t}\alpha_1^{2n_t}}     \cdot\prod_{l=1}^{\ell_t}\frac{1}{q_N^2(w_l^{(t)})}\right)\right].\end{equation*}
\item for $n=n_{k}$,
\begin{equation*}\label{ff9}|J_{n_{k}}(x)|\ge \frac{1}{2^8}\cdot \left(\frac{1}{\alpha_0^{n_{k}}\alpha_1^{n_{k}}}\cdot |J_{n_{k}-1}(x)|\right).\end{equation*}

\item for $n=n_{k}+1$, \begin{equation*}\label{ff10}|J_{n_{k}+1}(x)|\ge \frac{1}{2^8}\cdot \frac{1}{\alpha_0^{n_{k}}\alpha_1^{2n_{k}}}\cdot |J_{n_{k}-1}(x)|.
\end{equation*}
 \item for each $1\le \ell<\ell_{k+1}$,
 \begin{equation*}\label{ff12}|J_{n_k+1+\ell N}(x)|\ge \frac{1}{2^3}\cdot \left(\frac{1}{2^{2\ell}}\cdot \prod_{i=1}^{\ell}\frac{1}{q_N^2(w_i^{(k+1)})}\right)\cdot \left[\prod_{t=1}^{k}\left(\frac{1}{2^{2\ell_t+4}}\cdot
 \frac{1}{\alpha_0^{2n_t}\alpha_1^{2n_t}}\cdot\prod_{l=1}^{\ell_t}\frac{1}{q_N^2(w_l^{(t)})}\right)\right].\end{equation*}

\item for $n_k+1+(\ell-1)N<n<n_k+1+\ell N$ with $1\le \ell\le \ell_{k+1}$, \begin{equation*}\label{ff13}
|J_{n}(x)|\ge c\cdot |J_{n_k+1+(\ell-1)N}(x)|,
\end{equation*} where $c=c(M, N)$ is an absolute constant.
\end{itemize}
\end{lem}
\subsection{Mass distribution}\

We define a probability measure supported on the Cantor set $E$.  Still express an element $x\in E$ as \begin{align*}
x=[w_1^{(1)},\dots, w_{\ell_1}^{(1)}, &a_{n_1}, a_{n_1+1}, w_1^{(2)},\dots, w_{\ell_2}^{(2)}, a_{n_2},a_{n_2+1},\\
 &\dots, w_1^{(k)},\dots, w_{\ell_k}^{(k)}, a_{n_k},a_{n_k+1},\dots],
\end{align*} where $w\in \U$ and $$
\ \ \alpha_0^{n_k}\le a_{n_k}< 2\alpha_0^{n_k},\ \ \alpha_1^{n_k}\le a_{n_k+1}< 2\alpha_1^{n_k}\ \ {\text{for all}}\ k\ge 1.
$$

We define a measure $\mu$ along the basic intervals $J_n(x)$ containing $x$ as follows.
\begin{itemize}
\item Let $n\le n_1+1$.
\begin{itemize}
  \item for each $1\le \ell\le \ell_1$, define $$
  \mu\big(J_{N\ell}(x)\big)=\prod_{i=1}^{\ell}\frac{1}{q_N(w_i^{(1)})^{2s}\cdot
   \alpha_0^{sN}}.
  $$

  Because of the arbitrariness of $x$, this defines the measure on all basic cylinders of order $\ell N$.
  \item for each integer $n$ with $(\ell-1)N<n<\ell N$ for some $1\le \ell\le \ell_1$, define
  $$
  \mu\big(J_n(x)\big)=\sum_{J_{\ell N}\subset J_n(x)}\mu\big(J_{\ell N}\big),
  $$ where the summation is over all basic cylinders of order $\ell N$ contained in $J_{n}(x)$. This is designed to ensure the consistency of a measure
{and} defines the measure on the basic cylinders of order up to $n_1-1$.
        \item for each $0\le i\le 1$, define $$
  \mu\big(J_{n_1+i}(x)\big)=\prod_{j=0}^i\frac{1}{\alpha_j^n}\cdot \mu\big(J_{n_1-1}(x)\big)= \prod_{l=1}^{\ell_1}\frac{1}{q_N(w_l^{(1)})^{2s}\cdot \alpha_0^{sN}}\cdot\prod_{j=0}^i\frac{1}{\alpha_j^n}.
  $$

\end{itemize}

\item Let $n_{k}+1<n\le n_{k+1}+1$. Assume the measure of all basic intervals of order $n_{k}+1$ has been defined. \begin{itemize}
  \item for each $1\le \ell\le \ell_{k+1}$, define \begin{align}\label{ff7}
  \mu\big(J_{n_{k}+1+N\ell}(x)\big)&=\prod_{i=1}^{\ell}\frac{1}{q_N(w_i^{(k+1)})^{2s}\cdot \alpha_0^{sN}}\cdot \mu\big(J_{n_{k}+1}(x)\big)\nonumber\\
  &=\left[\prod_{i=1}^{\ell}\frac{1}{q_N(w_i^{(k+1)})^{2s}\cdot \alpha_0^{sN}}\right]\cdot \left[\prod_{t=1}^k\left(\frac{1}{\alpha_0^{n_t}\alpha_1^{n_t}}\prod_{l=1}^{\ell_t}\frac{1}{q_N^{2s}(w_l^{(t)})\cdot \alpha_0^{sN}}\right)\right].
  \end{align}

  \item for each integer $n$ with $n_k+1+(\ell-1)N<n<n_k+1+\ell N$ for some $1\le \ell\le \ell_{k+1}$, define
  $$
  \mu\big(J_n(x)\big)=\sum_{J_{n_k+1+\ell N}\subset J_n(x)}\mu(J_{n_k+1+\ell N}).
  $$

{This defines the measure on the basic cylinders of order up to $n_{k+1}-1$.}
  \item for each $0\le i\le 1$, define \begin{align}\label{ff11}
  \mu\big(J_{n_{k+1}+i}(x)\big)&=\prod_{j=0}^i\frac{1}{\alpha_j^{n_{k+1}}}\cdot \mu\big(J_{n_{k+1}-1}(x)\big).
  \end{align}
%
\end{itemize}
\end{itemize}
Look at (\ref{ff7}) for the measure of a basic cylinder of order $n_k+1+\ell N$ and its predecessor of order $n_k+1+(\ell-1)N$: the former has more one term than the latter, i.e. the term $$
\frac{1}{q_N^{2s}(w_{\ell}^{(k+1)})\alpha_0^{sN}}
$$    which is uniformly bounded by some constant depending on $M,N, B$. Thus there is an absolute constant $c=c(M,N,B)>0$ such that for each integer $n$ with \begin{itemize}
  \item when $n_k+1+(\ell-1)N\le n\le n_k+1+\ell N$,
\begin{equation*}\label{ff1}\mu\big(J_n(x)\big)\ge c\cdot \mu\big(J_{n_k+1+(\ell-1)N}(x)\big).\end{equation*}
\item
when $n_k+1\le n<n_{k+1}-1$, \begin{equation}\label{ff2}
\mu\big(J_{n+1}(x)\big)\ge c\cdot \mu\big(J_n(x)\big).
\end{equation}
\end{itemize}

\subsection{H\"{o}lder exponent of $\mu$ for basic cylinders}\

We begin with some simple relations between $A$ and $B$ and beyond. \begin{lem}\label{l2} Recall the real number $A$ and the integer $N$ given before in \eqref{fff1}. Then
  \begin{itemize}
  \item $\displaystyle
  \left(\frac{1}{\alpha_1\alpha_0^2}\right)^s=\frac{1}{\alpha_0}\cdot \frac{1}{\alpha_0^s}, \ \ {\text{equivalently}}\ \frac{1}{\alpha_0}=\left(\frac{1}{\alpha_0\alpha_1}\right)^{s};
  $
  \smallskip
  
  \item $\displaystyle
  \frac{1}{\alpha_0\alpha_1}\cdot \frac{1}{\alpha_0^s}\le \left(\frac{1}{\alpha_0^2\alpha_1^2}\right)^s\ \ {\text{equivalently}}\ \frac{1}{\alpha_0\alpha_1}\le \left(\frac{1}{\alpha_0\alpha_1^2}\right)^s;
  $
 \smallskip
  
   \item Let $\epsilon>0$. Then we can {choose an} integer $N$ so large and $\{\ell_k\}$ so sparse that $$
  2^{2\ell_k+8}\le \left(2^{(N-1)\ell_k}\right)^{\epsilon}\text{ {and} }\  \ell_kN\ge (1-\epsilon)n_k\ \ {\text{for all}}\ k\ge 1.
  $$
  \end{itemize}
\end{lem}
\begin{proof}
  Recall that we are in the case {when} $$
  f_{{t_0,t_1}}(s)=\frac{sf_{{{t_0}}}(s)}{{{t_1}}\left[f_{{{t_0}}}(s)+\frac{s}{{{t_1}}}-\frac{2s-1}{{{t_0}}}\right]}
  =\frac{sf_{{{t_0}}}(s)}{{{t_1}}\left[\frac{s}{{{t_1}}}+\frac{1-s}{{{t_0}}}\right]}.
  $$
  Thus, by recalling the choice of $A$, it follows that\begin{align*}
  \left(\frac{1}{\alpha_1\alpha_0^2}\right)^s=\frac{1}{\alpha_0}\cdot \frac{1}{\alpha_0^s}&\Longleftrightarrow \alpha_1^s=\alpha_0^{1-s}
  \Longleftrightarrow  \left(\frac{B}{A}\right)^{\frac{s}{{{t_1}}}}=A^{\frac{1-s}{{{t_0}}}}\\
  &\Longleftrightarrow  {\frac{s}{{{t_1}}}}\log B=\left({\frac{s}{{{t_1}}}+\frac{1-s}{{{t_0}}}}\right)\log A\\ &
  \Longleftrightarrow  {\frac{s}{{{t_1}}}}\log B=\frac{sf_{{{t_0}}}(s)}{{{t_1}}f_{{t_0,t_1}}(s)}\log A\\
  &\Longleftrightarrow f_{{t_0,t_1}}(s)\log B=f_{t_0}(s)\log A,
  \end{align*} where the last equality is just how $A$ was chosen.

  Substitute the first equality into the second claim, it is nothing but to say$$
  \alpha_1^s\le \alpha_1.
  $$
  The last claim is trivial.
\end{proof}

We compare the measure and length of $J_{n}(x)$.

(1) {Let} $n={n_{k}-1}$ which is equal to $n_{k-1}+1+\ell_{k}N$.
{Recall (\ref{ff7}) with $k$ replaced by $k-1$, and then take $\ell=\ell_k$}. Summing over the product on $\alpha_0$ and using the third item in Lemma \ref{l2}, it follows that \begin{align*}
\mu\big(J_{n_{k}-1}(x)\big)&\le \left(\frac{1}{\alpha_0^{n_{k}s}}\right)^{1-\epsilon}\left[\prod_{i=1}^{\ell_{k}}\frac{1}{q_N(w_i^{(k)})^{2s}}\right]\cdot \left[\prod_{t=1}^{k-1}\left(\Big(\frac{1}{\alpha_0^{n_t}\alpha_1^{n_t}\alpha_0^{n_t s}}\Big)^{1-\epsilon}\prod_{l=1}^{\ell_t}\frac{1}{q_N^{2s}(w_l^{(t)})}\right)\right].\end{align*}
Then using the second item in Lemma \ref{l2} by changing $\alpha_0^{1+s}\alpha_1$ to $(\alpha_0^2\alpha_1^2)^s$, one has
\begin{align*}
\mu\big(J_{n_{k}-1}(x)\big)&\le \left(\frac{1}{\alpha_0^{n_{k}s}}\right)^{1-\epsilon}\left[\prod_{i=1}^{\ell_{k}}\frac{1}{q_N(w_i^{(k)})^{2s}}\right]\cdot \left[\prod_{t=1}^{k-1}\left(\Big(\frac{1}{\alpha_0^{2n_t}\alpha_1^{2n_t}}\Big)^{s(1-\epsilon)}\prod_{l=1}^{\ell_t}\frac{1}{q_N^{2s}(w_l^{(t)})}\right)\right].
\end{align*}
At last, by the third item in Lemma \ref{l2}, we have $$
\prod_{l=1}^{\ell_t}\frac{1}{q_N^{2s}(w_l^{(t)})}\le \left(\frac{1}{2^{2\ell_t+8}}\cdot \prod_{l=1}^{\ell_t}\frac{1}{q_N^{2}(w_l^{(t)})}\right)^{s(1-\epsilon)}.
$$
Finally, by comparing with the length of $J_{n_{k}-1}(x)$ (Lemma \ref{l6.5}), we arrive at $$
\mu\big(J_{n_{k}-1}(x)\big)\le 8\cdot |J_{n_{k}-1}(x)|^{s(1-\epsilon)}.
$$

(2) {Let} $n={n_{k}}$.
Recall (\ref{ff11}). By the first item in Lemma \ref{l2}, \begin{align*}
\mu\big(J_{n_{k}}(x)\big)&=\frac{1}{\alpha_0^{n_{k}}}\cdot \mu\big(J_{n_{k}-1}(x)\big)\le 8\cdot \frac{1}{\alpha_0^{n_{k}}}\cdot |J_{n_{k}-1}(x)|^{s(1-\epsilon)}\\
&\le 8\cdot \left(\frac{1}{\alpha_0^{n_{k}}\alpha_1^{n_{k}}}\cdot \big|J_{n_{k}-1}(x)\big|\right)^{s(1-\epsilon)}.\end{align*}
By comparing with the length of $J_{n_{k}}(x)$ (Lemma \ref{l6.5}), we arrive at $$
\mu\big(J_{n_{k}}(x)\big)\le 2^{11}\cdot \big|J_{n_{k}}(x)\big|^{s(1-\epsilon)}.
$$

(3) {Let} $n=n_{k}+1$. Recall (\ref{ff11}). By the second item in Lemma \ref{l2}, \begin{align*}
\mu\big(J_{n_{k}+1}(x)\big)&=\frac{1}{\alpha_0^{n_{k}}\alpha_1^{n_{k}}}\cdot \mu\big(J_{n_{k}-1}(x)\big)\le 8\cdot \frac{1}{\alpha_0^{n_{k}}\alpha_1^{n_{k}}}\cdot \big|J_{n_{k}-1}(x)\big|^{s(1-\epsilon)}\\
&\le 8\cdot \left(\frac{1}{\alpha_0^{n_{k}}\alpha_1^{2n_{k}}}\cdot \big|J_{n_{k}-1}(x)\big|\right)^{s(1-\epsilon)}.\end{align*}
By comparing with the length of $J_{n_{k}+1}(x)$ (Lemma \ref{l6.5}), we arrive at $$
\mu\big(J_{n_{k}+1}(x)\big)\le 2^{11}\cdot \big|J_{n_{k}+1}(x)\big|^{s(1-\epsilon)}.
$$

(4) {Let} $n=n_k+1+\ell N$ for some $1\le \ell<\ell_{k+1}$. Compare Lemma \ref{l6.5} and the formula (\ref{ff7}). In (\ref{ff7}), {after deleting} the term $\alpha_0^{sN}$ in the first product  and {changing} $\alpha_0^{1+s}\alpha_1$ to $(\alpha_0^2\alpha_1^2)^s$ in the second product, we will arrive at $$
\mu\big(J_{n_k+1+\ell N}(x)\big)\le 2^{11}\cdot \big|J_{n_k+1+\ell N}(x)\big|^{s(1-\epsilon)}.
$$

(5) For other $n$, let $1\le \ell\le \ell_k$ be the integer such that $$
n_k+1+(\ell-1)N\le n<n_k+1+\ell N.
$$
Then \begin{align*}
  \mu\big(J_n(x)\big)\le \mu\big(J_{n_k+1+(\ell-1)N}(x)\big)\le 2^{11}\cdot \big|J_{n_k+1+(\ell-1)N}(x)\big|^{s(1-\epsilon)}\le 2^{11}\cdot c^{-1}\cdot \big|J_{n}(x)\big|^{s(1-\epsilon)},
\end{align*} where for the last inequality we have used Lemma \ref{l6.5} for the equivalence of the lengths of the two basic cylinders.

In a summary, we have show that for some absolute constant $c_1$, for any $n\ge 1$ and $x\in E$, \begin{align}\label{g1}
  \mu\big(J_n(x)\big)\le c_1\cdot \big|J_n(x)\big|^{s(1-\epsilon)}.
\end{align}

\subsection{H\"{o}lder exponent of $\mu$  for a general ball}
\

Recall Lemma \ref{l3} about the relation between the gap and the length of the basic cylinders:
$$
G_n(x)\ge \frac{1}{M}\cdot |J_n(x)|.
$$
We consider the measure of a general ball $B(x,r)$ with $x\in E$ and $r$ small.
Let $n$ be the integer such that $$
G_{n+1}(x)\le r<G_{n}(x).
$$ Then the ball $B(x,r)$ can only intersect one basic cylinder of order $n$, i.e. the basic cylinder $J_n(x)$, and so all the basic cylinders of order $n+1$ for which $B(x,r)$ can intersect are all contained in $J_n(x)$.

Let $k$ be the integer such that $$
n_{k-1}+1\le n\le n_{k}.
$$

(1)  {Let} $n_{k-1}+1\le n<n_{k}-1$. By (\ref{ff2}) and (\ref{g1}), it follows that
\begin{align*}
\mu\big(B(x,r)\big)&\le \mu\big(J_n(x)\big)\le c\cdot \mu\big(J_{n+1}(x)\big)\le c\cdot c_1\cdot \big|J_{n+1}(x)\big|^{s(1-\epsilon)}\\&\le c\cdot c_1\cdot M\cdot  \big(G_{n+1}(x)\big)^{s(1-\epsilon)}\le c\cdot c_1\cdot M\cdot  r^{s(1-\epsilon)}.
\end{align*}

(2) {Let} $n=n_{k}-1$. The ball $B(x,r)$ can only intersect the basic cylinder $J_{n_k-1}(x)$ of order $n_k-1$. Now we consider how many basic cylinders of order $n_k$ contained in $J_{n_k-1}(x)$ and with non-empty intersecting with the ball $B(x,r)$.

We write a basic cylinder of order $n_k$ contained in $J_{n_k-1}(x)$ as $$
J_{n_k}(u, a), \ {\text{for some}}\ \alpha_0^{n_k}\le a<2\alpha_0^{n_k}.
$$
It is trivial that for each $a$, the basic cylinder $J_{n_{k}}(u,a)$ is contained in the cylinder $I_{n_k}(u,a)$ and the latter interval is of length $$
\frac{1}{q_{n_k}(q_{n_k}+q_{n_k-1})}\ge \frac{1}{8}\cdot \frac{1}{q_{n_k-1}^2(u)\alpha_0^{2n_k}}.
$$

\begin{itemize}
  \item {Let} $$
  r<\frac{1}{8}\cdot\frac{1}{q_{n_k-1}^2(u)\alpha_0^{2n_k}}.
  $$ Then the ball $B(x,r)$ can intersect at most three cylinders $I_{n_k}(u,a)$ and so three basic cylinders $J_{n_k}(u,a)$. Note that all those basic cylinder are of the same $\mu$-measure, thus \begin{align*}
  \mu\big(B(x,r)\big)&\le 3\mu\big(J_{n_k}(x)\big)\le 3 \cdot c_1\cdot |J_{n_k}(x)|^{s(1-\epsilon)}\\
  &\le 3\cdot c_1\cdot M\cdot G_{n+1}(x)^{s(1-\epsilon)}\le 3\cdot c_1\cdot M\cdot r^{s(1-\epsilon)}.
  \end{align*}

\item {Let} $$
  r\ge \frac{1}{8}\cdot\frac{1}{q_{n_k-1}^2(u)\alpha_0^{2n_k}}.
  $$ The number of cylinders $I_{n_k}(u,a)$ for which the ball $B(x,r)$ can intersect is at most $$
  {16r}\cdot q_{n_k-1}^{2}(u)\alpha_0^{2n_{k}}+2\le 2^5\cdot {r}\cdot q_{n_k-1}^{2}(u)\alpha_0^{2n_{k}},
  $$ so at most this number of basic cylinders of order $n_k$ for which the ball $B(x,r)$ can intersect. Thus \begin{align*}
    \mu\big(B(x,r)\big)&\le \min\Big\{\mu\big(J_{n_k-1}(x)\big),\ \  2^5\cdot {r}\cdot q_{n_k-1}^{2}(u)\alpha_0^{2n_{k}}\cdot \frac{1}{\alpha_0^{n_k}}\cdot \mu\big(J_{n_k-1}(x)\big)\Big\}\\
    &\le c_1\cdot |J_{n_k-1}(x)|^{s(1-\epsilon)}\cdot \min\Big\{1, 2^5\cdot {r}\cdot q_{n_k-1}^{2}(u)\alpha_0^{n_{k}}\Big\}\\
    &\le c_1\cdot \left(\frac{1}{q_{n_k-1}(u)^2 \alpha_0^{n_k}}\right)^{s(1-\epsilon)}\cdot 1^{1-s(1-\epsilon)}\cdot \Big(2^5\cdot {r}\cdot q_{n_k-1}^{2}(u)\alpha_0^{n_{k}}\Big)^{s(1-\epsilon)}\\
    &=c_2 \cdot r^{s(1-\epsilon)}.
  \end{align*}
\end{itemize}

(3) {Let} $n=n_k$. By changing $n_k-1$ and $\alpha_0$ in case (2) to $n_k$ and $\alpha_1$ respectively and then following the same argument as in case (2), we can arrive the same conclusion.

\subsection{Conclusion}\

Thus by applying the mass distribution principle (Proposition \ref{p1}), it yields that $$
\hdim E\ge s(1-\epsilon).
$$ Since $E\subseteq  {\mathcal{E}_{\bold{t}}(B)}$ and $\epsilon, s$ are arbitrary, we conclude that $$
\hdim {\mathcal{E}_{\bold{t}}(B)}\ge s_o.
$$

\section{Completing the proof of Theorem \ref{BHWthm}}

{\sc Upper bound}. For any $\epsilon>0$, one has $$
\Psi(n)\ge (B-\epsilon)^n  \ \ {\text{for all}}\ n\gg 1.
$$  Thus $$
{\mathcal{E}_{\bold{t}}(\Psi)}\subset \Big\{x\in [0,1): a_{n}^{{{t_0}}}(x)a_{n+1}^{{{t_1}}}(x)\ge (B-\epsilon)^n, \ {\text{i.m.}}\ n\in \N\Big\}.
$$ Therefore, $$
\hdim \mathcal{E}_{\bold{t}}(\Psi)\le s_o(B-\epsilon).
$$
Recall Proposition \ref{tb} for the continuity of $s_o=s_o(B)$ with respect to $B$. Then by letting $\epsilon\to 0$, the upper bound {for} $\hdim \mathcal{E}_{\bold{t}}(\Psi)$ follows.

{{\sc Lower bound}. The argument for the lower bound of $\mathcal{E}_{\bold{t}}(\Psi)$ is almost the same as for $\mathcal{E}_{\bold{t}}(B)$ given in last section. So we only give the outline of the proof and mark some minor differences.

Recall the definition of $s_o(B)$ and Lemma \ref{l6.2}. If $\frac{s_o(B)}{{{t_1}}}-\frac{2s_o(B)-1}{{{t_0}}}\le 0$, then by Theorem \ref{WaWu} and Lemma \ref{l6.2}  it follows that $$
\hdim \mathcal{E}_{\bold{t}}(\Psi)\ge \hdim \Big\{x\in [0,1): a_{n+1}^{t_1}(x)
\ge \Psi(n), \ {\text{i.m.}}\ n\in \N\Big\}=s_o(B).
$$
Then we are in the remaining case {when} \begin{equation}\label{1}
\frac{s_o(B)}{{{t_1}}}-\frac{2s_o(B)-1}{{{t_0}}}>0.
\end{equation}}
{At first, choose a real number {$\widetilde{B}>B$ 
close enough to} $B$ such that (\ref{1}) is still true when replacing $B$ by $\widetilde{B}$. Secondly fix integers $M,N$ sufficiently large such that $s:=s^{(2)}_N\big(M,\widetilde{B}\big)$ is in {a small enough} neighborhood of $s_o(\widetilde{B})$ so that $1>s>1/2$, and (\ref{ff5}) holds. At last define a real number $\widetilde{A}$ such that
\begin{equation*}
f_{{{t_0}}}(s)\log \widetilde{A}=f_{{t_0, t_1}}(s)\log \widetilde{B}.
\end{equation*}

By the definition of $B$,
one can choose a sparse {enough} sequence of integers $\{n_k\}_{k\ge 1}$ such that $$
\Psi(n_k)\le \widetilde{B}^{n_k}  \ {\text{for all}}\ k\ge 1.
$$  Thus $$
\mathcal{E}_{\bold{t}}(\Psi)\supset \Big\{x\in [0,1): a_{n_k}^{{{t_0}}}(x)a_{n_k+1}^{{{t_1}}}(x)\ge \widetilde{B}^{n_k}  \ {\text{for all}}\ {k\ge 1}\Big\}.
$$}
{So we are almost in the same situation as {when} proving the lower bound {for} $\hdim \mathcal{E}_{\bold{t}}(B)$.
The only difference,  
{besides the notational differences $(A,B) \mapsto (\widetilde{A}, \widetilde{B})$}, is that the number of the integers in the interval $(n_k+1, n_{k+1})$ may not be a multiple of $N$. }

{Therefore, for all $k\ge 1$, write {(shifting the indices from $n_0+1$ to $0$)}
 $$
(n_{k}-1)-(n_{k-1}+1)=\ell_{k}N+i_{k} \ {\text{for some}}\ 0\le i_{k}<N,
$$ and define a Cantor subset of $\mathcal{E}_{\bold{t}}(\Psi)$ as
\begin{align*}
  \widetilde{E}
=\Bigg\{x\in [0,1): \widetilde{A}^{\frac{n_k}{{{t_0}}}}\le a_{n_k}(x)&<2 {\widetilde{A}^{\frac{n_k}{{{t_0}}}}}, \ \left(\frac{\widetilde{B}^{n_k}}{\widetilde{A}^{n_k}}\right)^{1/{{t_1}}}\le a_{n_k+1}(x)<2\left(\frac{\widetilde{B}^{n_k}}{\widetilde{A}^{n_k}}\right)^{1/{{t_1}}} {\text{for all}} \ k\ge 1;
\\ & a_{n_k+2}(x)=\cdots =a_{n_k+1+i_{k+1}}(x)=2 \ {\text{for all}} \ k\ge 0;
\\ &{\text{and}}\ a_n(x)\in \{1,\dots, M\}  \ {\text{for other $n\in \N$}}\Bigg\}.
\end{align*}

Use the same notation as in Section \ref{s6}: $$
\U=\{w=(\sigma_1,\dots, \sigma_N): 1\le \sigma_i\le M, \ 1\le i\le N\}
$$ and $$
 \alpha_0=\widetilde{A}^{1/{{t_0}}}, \ \  \alpha_1=\left(\frac{\widetilde{B}}{\widetilde{A}}\right)^{1/{{t_1}}},
  $$
  and {define} $J_n(x)$ 
  in the same way.

A generic
element $x\in \widetilde{E}$ can be written as \begin{align*}
x=\big[\eta^{(1)}, w_1^{(1)},\dots, w_{\ell_1}^{(1)}, a_{n_1}, a_{n_1+1},\ & \eta^{(2)}, w_1^{(2)},\dots, w_{\ell_2}^{(2)}, a_{n_2},a_{n_2+1},\\
 \dots, \ & \eta^{(k)}, w_1^{(k)},\dots, w_{\ell_k}^{(k)}, a_{n_k},a_{n_k+1},\dots\big],
\end{align*} where $\eta^{(k)}=(\underbrace{2, \dots, 2}_{i_k})$, $w\in \U$, and $$
\alpha_0^{n_k}\le a_{n_k}< 2\alpha_0^{n_k}, \ \alpha_1^{n_k}\le a_{n_k+1}< 2\alpha_1^{n_k}\ \ {\text{for all}}\ k\ge 1.
$$

Recall {that} $s=s^{(2)}_N(M, \widetilde{B})$. We define the measure {of} the basic intervals $J_n(x)$ containing $x$ as follows. Note that for all $x\in \widetilde{E}$ their partial quotients $a_n(x)$ have only one choice for all $$
n_k+1<n\le n_k+1+i_{k+1}, \ {\text{with}}\ k\ge 0.
$$ So when defining a mass distribution $\mu$ on $\widetilde{E}$, one must have that$$
\mu\big(J_n(x)\big)=\mu\big(J_{n_k+1}(x)\big), \ {\text{for all}}\ n_k+1<n\le n_{k}+1+i_{k+1}.
$$ Except such a restriction, we define the measure on $\widetilde{E}$ in the way as did in Section \ref{s6}:
%
%
%
%

Let $n_{k}+1<n\le n_{k+1}+1$. Assume the measure of all basic intervals of order $n_{k}+1$ has been defined. \begin{itemize}
  \item for each $n_k+1<n\le n_k+1+i_{k+1}$, define $$
  \mu\big(J_n(x)\big)=\mu\big(J_{n_k+1}(x)\big),
  $$
  \item for each $1\le \ell\le \ell_{k+1}$, define \begin{align*}
  \mu\big(J_{n_{k}+1+i_{k+1}+N\ell}(x)\big)&=\prod_{i=1}^{\ell}\frac{1}{q_N(w_i^{(k+1)})^{2s}\cdot \alpha_1^{sN}}\cdot \mu\big(J_{n_{k}+1+i_{k+1}}(x)\big).
  \end{align*}

  \item for each integer $n$ with $n_k+1+i_{k+1}+(\ell-1)N<n<n_k+1+i_{k+1}+\ell N$ for some $1\le \ell\le \ell_{k+1}$, define
  $$
  \mu\big(J_n(x)\big)=\sum_{J_{n_k+1+i_{k+1}+\ell N}\subset J_n(x)}\mu(J_{n_k+1+i_{k+1}+\ell N}).
  $$
  \item for each $0\le i\le 1$, define \begin{align*}
  \mu\big(J_{n_{k+1}+i}(x)\big)&=\prod_{j=0}^i\frac{1}{\alpha_j^{n_{k+1}}}\cdot \mu\big(J_{n_{k+1}}(x)\big).
  \end{align*}
%
\end{itemize}

Then we will use the mass distribution principle (Proposition \ref{p1}) to reach our conclusion that
$$\hdim \mathcal{E}_{\bold{t}}(\Psi)\ge s_o(B).$$
So the {remaining} task is to compare the $\mu$-measure of a ball $B(x,r)$ with $r$. The gap estimation (Lemma \ref{l3}) is still true without any change and the estimation on $q_n(x)$ (Lemma \ref{l6.4}) is similar just by adding some terms of the power of 2.  
Then the remaining argument can 
{proceed as} 
in Section \ref{s6} with some obvious modifications. 
We omit the details.}


{\color{blue}}

\section{Final Comments}\label{final}
One might wonder to extend Theorem \ref{BHWthm} to all $m\geq 2$. Our methods for the upper bound calculations extend easily to any $m$, but the major difficulty lies in establishing the lower bound and proving that it is equal to the upper bound estimate. To be precise, it is possible to prove the following formula:

\begin{thm}\label{BHWthm2}  
Let $\Psi:\N\to\R_{\ge 1}$ 
{be such that $1<B< \infty$. Then}
\begin{equation*}
\begin{array}{ll} \hdim {\D_{\mathbf t}}(\Psi ) & \leq \inf \{s\ge 0: P(T, -s\log |T'|-f_{{{t_0,\dots, t_{m-1}}}}(s)\log B)\le 0\}  \ \mathrm{for \ all} \ m,   \ but \\ [3ex] \hdim {\D_{\mathbf t}}(\Psi )& {\ge} \inf \{s\ge 0: P(T, -s\log |T'|-f_{{t_0,t_1}}(s)\log B)\le 0\},
\end{array}
\end{equation*}

where $f_{{{t_0,\dots, t_{m-1}}}}(s)$ is given by the following iterative {procedure} with the starting value as ${f_{t_0}(s)=\frac{s}{t_0}},$ and {\begin{align*}
f_{t_0,\dots,t_{\ell}}(s)&=\frac{sf_{t_0,\dots, t_{\ell-1}}(s)}{t_{\ell}f_{t_0,\dots,t_{\ell-1}}(s)+\max\left\{0, s-(2s-1)\frac{t_{\ell}}{t_i}, 0\le i\le \ell-1\right\}}\\
&=\frac{s{f_{t_0,\dots, t_{\ell-1}}(s)}}{t_{\ell}{f_{t_0,\dots,t_{\ell-1}}}(s)+\max\left\{0, s-(2s-1)\frac{t_{\ell}}{\max_{0\le i\le \ell-1}t_i}\right\}}.
\end{align*}}

\end{thm}
We believe that, for $1<B<\infty,$  $$\hdim {\D_{\mathbf t}}(\Psi ) \geq \inf \{s\ge 0: P(T, -s\log |T'|-f_{{t_0,\dots, t_{m-1}}}(s)\log B)\le 0\}$$
should hold. From the definition of the functions $f_{{t_0,\dots, t_{\ell-1}}}$, the 
{appearance of the expression \linebreak $\max_{0\le i\le \ell-1}t_i$} in it means that {the partial quotients corresponding  to some exponents} $t_i$ 
will not contribute to the dimension. Thus the major difficulty is to figure out which partial quotients contribute essentially to the dimension and which are not.

\providecommand{\bysame}{\leavevmode\hbox to3em{\hrulefill}\thinspace}
\providecommand{\MR}{\relax\ifhmode\unskip\space\fi MR }
\providecommand{\MRhref}[2]{%
  \href{http://www.ams.org/mathscinet-getitem?mr=#1}{#2}
}
\providecommand{\href}[2]{#2}

\end{document}